\documentclass[11pt]{amsart}

\usepackage{amsmath,amssymb,amsthm}
\usepackage{graphicx}
\usepackage{cite}
\usepackage{mathtools}
\usepackage{xurl}
\usepackage[usenames, dvipsnames]{color}
\usepackage{hyperref}
\hypersetup{
    colorlinks=true,
    linkcolor=blue,
    filecolor=magenta,      
    urlcolor=cyan,
    pdftitle={Overleaf Example},
    pdfpagemode=FullScreen,
    }

\normalsize

\def\dfrac#1#2{\lower0.15ex\hbox{\large$\frac{#1}{#2}$}}

\allowdisplaybreaks

\def\non{\nonumber}

\def\cB{\mathcal{B}}

\def\cF{\mathcal{F}}

\def\cM{\mathcal{M}}

\def\eps{\varepsilon}
\def\N{\mathbb{N}}
\def\R{\mathbb{R}}
\def\C{\mathbb{C}}
\def\Q{\mathbb{Q}}
\def\Z{\mathbb{Z}}

\def\g{\mathfrak{g}}

\def\E{\mathbb{E}}

\def\h{\mathfrak{h}}
\def\cX{\mathcal{X}}
\def\cY{\mathcal{Y}}
\def\cZ{\mathcal{Z}}
\def\nil{\mathrm{nil}}
\def\unf{\mathrm{unf}}
\def\sml{\mathrm{sml}}
\def\mal{\mathrm{Mal}}

\DeclareMathOperator\spa{span}

\DeclareMathOperator\Hom{Hom}

\DeclareMathOperator\poly{poly}

\newtheorem{firstthm}{Proposition}[section]
\newtheorem{thm}[firstthm]{Theorem}
\newtheorem{prop}[firstthm]{Proposition}

\newtheorem{lemma}[firstthm]{Lemma}
\newtheorem{cor}[firstthm]{Corollary}

\theoremstyle{definition}
\newtheorem{example}[firstthm]{Example}
\newtheorem{definition}[firstthm]{Definition}
\newtheorem{remark}[firstthm]{Remark}
\newtheorem{thd}[firstthm]{Definition/Theorem}

\newcommand{\vect}[1]{\boldsymbol{#1}}

\title{A non-flag arithmetic regularity lemma and counting lemma}

\author{Daniel Altman}
\address{University of Oxford, Mathematical Institute, Radcliffe Observatory Quarter, Woodstock Rd, Oxford, OX2 6GG, United Kingdom}
\email{daniel.h.altman@gmail.com}

\begin{document}

\maketitle

\begin{abstract}
Green and Tao's arithmetic regularity lemma and counting lemma together apply to systems of linear forms which satisfy a particular algebraic criterion known as the `flag condition'. We give an arithmetic regularity lemma and counting lemma which applies to all systems of linear forms.
\end{abstract}

\setcounter{tocdepth}{1}
\tableofcontents

\section{Introduction}

\subsection{Context and summary}
The use of `regularity lemmas' in combinatorics dates back to Szem\'eredi's proof of his eponymous theorem on arithmetic progressions in dense sets of integers \cite{Sz75}. Szemer\'edi's regularity lemma and developments thereof have since become powerful tools in combinatorics which have been brought to bear on a far wider range of problems than the study of arithmetic progressions in sets of integers. Where Szemer\'edi's regularity lemma is a structural result about graphs, Green introduced an arithmetic analogue of this brand of regularity results in \cite{G05}, and used it in that paper to solve some problems in additive combinatorics. From the perspective of applications to additive patterns in the integers, the state of the art is Green and Tao's arithmetic regularity lemma \cite{GT10}.

Green and Tao's arithmetic regularity lemma is phrased as a decomposition result for bounded functions which consequently may be written as the sum of a nilsequence, a Gowers uniform function, and an error. \cite{GT10} demonstrates a number of applications for which the contributions of the uniform and error components are negligible and so the problem is essentially reduced to an analysis of nilsequences. Complementing the arithmetic regularity lemma in \cite{GT10} is an arithmetic counting lemma which gives a general formula (via equidistribution on nilmanifolds) for computing averages on linear patterns in the nilsequence component. 

It has come to light (see \cite{GT20}, \cite{T20}) that this counting lemma \cite[Theorem 1.11]{GT10} applies only to systems of linear forms which satisfy the so-called \textit{flag condition} (see Subsection \ref{ss:flag} for definitions), whence the arithmetic regularity and counting lemma strategy may be exercised only on this particular set of linear patterns. Further to this, there are examples (see Subsection \ref{ss:examples}) which demonstrate that the output of the arithmetic regularity lemma \cite[Theorem 1.2]{GT10}  is too coarse to determine the distribution of non-flag linear patterns in any prospective counting lemma.

The goal of this paper is to give an arithmetic regularity lemma and complementary counting lemma which are able to handle all linear patterns. The main results are a stronger arithmetic regularity lemma Theorem \ref{t:strong-arl}, which feeds into a new counting lemma Theorem \ref{t:main-quant}, which itself is able to handle all sets of linear forms. We remark that our description of the nilmanifold on which we obtain equidistribution in the counting lemma is less explicit than what is worked out for the Leibman group in \cite[Section 3]{GT10} (though this is partially remedied in Section \ref{s:gPsiSbullet} below). Nonetheless, the broadstrokes `arithmetic regularity and counting' strategy for solving problems in additive combinatorics can now be applied to all linear patterns. 

We also give a couple of applications; these are discussed in the next subsection. 

\subsection{Results}

Given a flag system of linear forms $\Psi$, a filtered Lie group $G_\bullet$ and a polynomial sequence $g:\Z \to G$ which is suitably `irrational' with respect to this filtration, Green and Tao's counting lemma determines (quantitatively) the distribution of the tuple $g^\Psi(\vect n):=((g(\psi_1(\vect n)), \ldots, g(\psi_t(\vect n)))\Gamma^t)_{\vect n \in \Z^D}$ in $G^t/\Gamma^t$ from the data $\{G_\bullet, \Psi\}$. Unfortunately, this data is insufficient to determine the distribution of $g^\Psi$ in the general, non-flag setting in the sense that there exist a non-flag system of linear forms $\Psi$ and distinct polynomial sequences $g$, $\tilde g$, both of which are irrational in the same $G_\bullet$, such that $g^\Psi$ distributes differently to $\tilde g^\Psi$; see Examples \ref{ex:flag-counter} and \ref{ex:heis-counter}. Instead of dealing with a filtered Lie group $G_\bullet$, we will consider a sequence of subspaces $S_\bullet$ of the Lie algebra $\g$ of $G$. One may define a notion of irrationality -- which we will call `linear irrationality' --  with respect to a sequence of subspaces such that for all linearly irrational polynomial sequences $g$ with respect to $S_\bullet$, one may determine the distribution of $g^\Psi$ from the data $\{S_\bullet, \Psi\}$. It will prove convenient to describe the nilmanifold on which $g^\Psi$ equidistributes as a Lie subalgebra of $\g^t$. We will denote this Lie subalgebra by $\g^\Psi(S_\bullet)$, define it in terms of $\Psi$ and $S_\bullet$ in Definition \ref{d:gPsiBullet} and then compute it more explicitly in Section \ref{s:gPsiSbullet}. 

Working in the Lie algebra, we will denote group multiplication by $\ast$. 

\begin{definition}
Define the operation $\ast:\g \times \g \to \g$ by $x\ast y := \log(\exp(x)\exp(y))$.
\end{definition}

We denote group multiplication in this way partially to avoid the temptation to accidentally distribute over addition. By Baker-Campbell-Hausdorff, we have explicitly that
\[x \ast y = x + y + \frac{1}{2}[x,y] + \frac{1}{12}[x,[x,y]] +  \ldots. \]

We will now state our counting lemma. Although there are a number of requisite definitions which have not yet been provided (rational basis: Definition \ref{d:rat-basis}, rational subspace: Definition \ref{d:rat-subspace}, complexity of a rational subspace: Definition \ref{d:comp-subspace}, linear irrationality: Definitions \ref{d:li} and \ref{d:lin-irrat-poly}, quantitative equidistribution: Definitions \ref{d:quant-equid-grp} and \ref{d:quant-equid-alg}, Hermite basis: Definition \ref{d:hermite-basis}), we hope that the reader may obtain a rough appreciation of the content of the result. It is proven as Theorem \ref{t:pr-main-quant} in Section \ref{s:cl}.  

\begin{thm}[Counting lemma]\label{t:main-quant}
Let $A,M\geq 2$ and $d,s,t\geq 1$ be integers. Let $\Psi = (\psi_1, \ldots, \psi_t)$ be a collection of linear forms each mapping $\Z^D$ to $\Z$. Let $\g$ be a real nilpotent Lie algebra with rational basis $\cX$, dimension $d$ and rational structure constants of height at most $M$. Let $S_\bullet$ be a sequence of rational subspaces of complexity  $M$ in $\g$. Suppose that $p(n)$ is an $(A,N)$-linearly irrational polynomial sequence  of degree $s$ in $S_\bullet$. Then there are constants $0<c_1,c_2 = O_{d,s,D,t,\Psi}(1)$ such that the polynomial sequence $p^\Psi(\vect x):= (p(\psi_1(\vect x)), \ldots, p(\psi_t(\vect x))$ is $O(M^{c_1}/A^{c_2})$-equidistributed on $\g^\Psi(S_\bullet)$ with respect to $\cX^\Psi$, the Hermite basis for $\g^\Psi(S_\bullet)$ in $(\g^t, \cX^t)$.
\end{thm}

The next task is to obtain a factorisation result for the notion of linear irrationality. The following theorem says that any polynomial sequence may be written as the product of an `error' sequence, a sequence which is linearly irrational with respect to some sequence of subspaces $S_\bullet$, and a rational sequence. Furthermore, the factorisation process is such that if we begin with a polynomial sequence which is irrational in the Green-Tao sense (we will henceforth call this `filtration irrational'), then the linearly irrational factorisation is still (though slightly less) Green-Tao (i.e. filtration) irrational. There are a number of additional definitions (filtration: Definition \ref{d:filtration}, polynomial sequence adapted to a filtration: Definition \ref{d:adapted},  quantitative rationality: Definition \ref{d:quant-rat}, quantitative smallness: Definition \ref{d:quant-small}, quantitative filtration irrationality: Definitions \ref{d:fi} and \ref{d:fi-poly}) required for a rigorous interpretation.

\begin{thm}[Factorisation of polynomial sequences]\label{t:state-factorise}
Let $A,M\geq 2$ and $d,s,t\geq 1$ be integers such that $AM> d$. Let $\g$ be a real nilpotent Lie algebra with filtration $\g_\bullet$, rational basis $\cX$, dimension $d$ and rational structure constants of height at most $M$. Let $p$ be a polynomial sequence adapted to $\g_\bullet$ and sequence of subspaces $S_\bullet\leq \g_\bullet$. Finally, insist that $S_\bullet, \g_\bullet$ are of complexity $A^{O(1)}$ in $\g$, where here and henceforth $O(1)$ terms may depend on $d,s$. Then we may write $p = e \ast p' \ast r$ where $e,p',r$ are polynomial sequences in $\g$ adapted to $\g_\bullet$, $p'$ is $(A,N)$-linearly irrational in a sequence of subspaces $S'_\bullet$ of complexity $A^{O(1)}$, $r$ is $(AM)^{O(1)}$-rational, and $e$ is $((AM)^{O(1)},N)$-small. 
\end{thm}
This Theorem is proven in Section \ref{s:fact} as Theorem \ref{t:quant-factorise} (where we actually prove a bit more about the factorisation -- see Theorem \ref{t:quant-factorise}). Our final goal is to provide an arithmetic regularity lemma which produces a virtual nilsequence which is both filtration irrational and linearly irrational; we call such a nilsequence `strongly irrational' (cf. Definition \ref{d:strong-irrat}). Though filtration irrationality is not strictly necessary to determine distribution, it is useful in applications and furthermore means that our arithmetic regularity lemma is more compatible with, and explicitly strengthens, what is done in \cite{GT10}. 

\begin{thm}[Strongly irrational arithmetic regularity lemma]\label{t:strong-arl}
Let $f:[N]\to [0,1]$, let $s \geq 1$, let $\eps > 0$, and let $\cF: \R^+ \to \R^+$ be a growth function. Then there exists a quantity $M=O_{s,\eps,\cF}(1)$ and a decomposition
\[f = f_\nil + f_\sml + f_\unf \]
of $f$ into functions $f_\nil, f_\unf: [N] \to [-1,1]$ such that
\begin{enumerate}
\item($f_\nil$ structured) $f_\nil$ is a degree $\leq s$, $(\cF(M),N)$-strongly irrational virtual polynomial nilsequence of complexity $\leq M$,
\item($f_\sml$ small) $||f_\sml||_{L^2[N]} \leq \eps$,
\item($f_\unf$ very uniform) $||f_\nil||_{U^{s+1}[N]} \leq 1/\cF(M)$,
\item(Nonnegativity) $f_\nil$ and $f_\sml$ take values in $[0,1]$.
\end{enumerate}
\end{thm}

We also provide a couple of applications. The first is the modest task of recovering (with polynomial bounds) the flag counting lemma for strongly irrational nilsequences.

\begin{thm}[Equidistribution in the Leibman group for flag systems]\label{t:flag-cl}
Let $A,M \geq 2$. Let $\Psi = (\psi_1, \ldots, \psi_t)$ be a flag system of linear forms, each mapping $\Z^D$ to $\Z$. Let $\g$ be a real nilpotent Lie algebra of dimension $d$ and step $s$. Let $\cX$ be a rational basis for $\g$ which passes through $\g_\bullet$ and with respect to which $\g$ has rational structure constants of height at most $M$. Let $S_\bullet\leq \g_\bullet$ be a sequence of subspaces of complexity at most $M$. If $p$ is $(A,N)$-strongly irrational in $(\g_\bullet,S_\bullet)$ then there are constants $0<c_1,c_2 = O_{d,s,D,t,\Psi}(1)$ such that  $p^\Psi$ is $O(M^{c_1}/A^{c_2})$-equidistributed in $\log G^\Psi$, where $G^\Psi$ is the Leibman group for $G_\bullet, \Psi$.
\end{thm}
This is proven in Subsection \ref{ss:recover}, where we show that under appropriate assumptions, $\g^\Psi(S_\bullet) = \log G^\Psi$.

Next, in Subsection \ref{ss:gw} we give another resolution of a conjecture of Gowers and Wolf \cite[Conjecture 2.5]{GW10} which was recently recovered in the generality of all linear forms in \cite{A21} by employing a somewhat specialised linear algebraic trick. The point of Subsection \ref{ss:gw} is to provide a less ad-hoc argument. Indeed, Green and Tao's proof in the flag setting \cite[Theorem 1.13]{GT10} may be followed after doing some modest algebraic computations (Proposition \ref{p:gw-equid}) and employing the strongly irrational arithmetic regularity lemma Theorem \ref{t:strong-arl} and counting lemma Theorem \ref{t:main-quant} in place of the original irrational arithmetic regularity lemma \cite[Theorem 1.2]{GT10} and counting lemma \cite[Theorem 1.11]{GT10} respectively. 

\begin{thm}[Gowers-Wolf Conjecture]
Let $\Psi = (\psi_1, \ldots, \psi_t)$ be a collection of linear forms each mapping $\Z^D$ to $\Z$, and let $s \geq 1$ be an integer such that the polynomials $\psi_1^{s+1},\ldots,\psi_t^{s+1}$ are linearly independent. For $i=1, \ldots t$,  let $f_i:[-N,N] \to \C$ be functions bounded in magnitude by 1 (and defined to be zero outside of $[-N,N]$). For all $\eps >0$ there exists $\delta>0$ such that if $\min_{i} ||f_i||_{U^{s+1}[-N,N]} \leq \delta$, then 
\[\left|\E_{\vect x \in [-N,N]^D} \prod_{i=1}^t f_i(\psi_i(\vect x))\right| \leq \eps.\]
\end{thm}

\subsection{On the flag condition}\label{ss:flag}
Let $\Psi=(\psi_1,\ldots, \psi_t)$ where each $\psi_i$ maps $\Z^D$ to $\Z$. For each $j \in \N^+$, one may write $\Psi^j(\vect{x}) = \sum_{m\in \cM}v_{m,j}m(\vect{x})$, where $v_{m,j} \in \Z^t$, and  $m \in \cM$  is a `monomial map' which maps $\vect x=(x_1,\ldots, x_D)$ to some monomial in the variables $\{x_1,\ldots, x_D\}$. To elucidate the notation, we compute with the specific example 
\[ \Psi(x,y) := (\psi_1(x,y), \psi_2(x,y), \psi_3(x,y), \psi_4(x,y)) := (y, 2x + 2y, x+3y, x),\]
that 
\[ \Psi^1(x,y) = (0,2,1,1)x+(1,2,3,0)y,\]
and 
\[ \Psi^2(x,y) = (0,4,1,1)x^2 + (1,4,9,0)y^2 + (0,8,6,0)xy.\]
For $i \in \N^+$ we define the vector spaces $V_i = \spa_\R\{v_{m,i}\}_{m \in \cM}$ and let $V=V_1$. In fact, it is not difficult to see that $V_i = V^i$ where $V^i$ is the smallest vector space containing $i$-fold products of elements of $V$, where multiplication is conducted coordinatewise.

A system of linear forms $\Psi$ satisfies the \textit{flag condition} if $V^i \leq V^j$ whenever $i\leq j$. One may check that the system of linear forms in the example above does not satisfy the flag condition (i.e. \textit{is not flag}) because $V$ is not contained in $V^2$.

In what follows, if context is sufficiently clear, we will use $V$ to denote the vector subspace of $\R^t$ formed in this way from $\Psi$. If context is less clear, we may also sometimes use the notation $V_\Psi$.

\subsection{Outline of the paper}
In Section \ref{s:examples} we begin by providing examples which motivate the transition from Lie group to Lie algebra. Then we work out the qualitative theory in the Lie algebra setting, giving a qualitative counting lemma Theorem \ref{t:main-qual} and qualitative factorisation theorem Proposition \ref{p:qual-factorise}. We encourage a reader who wishes to understand the ideas behind Theorems \ref{t:main-quant} and \ref{t:state-factorise} to at least glean from this section the main algebraic strategy before progressing to the quantitative sections that follow. The theory is substantially cleaner  here and we hope that this section convinces the reader that this is the right setup in general. This section may also be of independent interest. 

The qualitative counting lemma in Section \ref{s:examples} describes a particular Lie algebra $\g^\Psi(S_\bullet)$ on which we obtain equidistribution. This is the same Lie algebra that appears in the statement of Theorem \ref{t:main-quant}. In Section \ref{s:gPsiSbullet} we give more explicit and tenable algebraic descriptions of this Lie algebra.

In Section \ref{s:cl}, we prove our quantitative counting lemma, Theorem \ref{t:main-quant}.

In Section \ref{s:fact}, we prove our quantitative factorisation of polynomial sequences for linear irrationality, Theorem \ref{t:state-factorise}. 

In  Section \ref{s:arl}, we prove our quantitative `strongly irrational' arithmetic regularity lemma, Theorem \ref{t:strong-arl}. 

In Section \ref{s:app} we give applications as discussed above. \newline \newline 
\textbf{Acknowledgments}. The author is grateful to Ben Green for many helpful conversations. The author is also grateful to Terence Tao for helpful comments on an earlier version of this document.

\section{Examples, discussion of methods and qualitative results}\label{s:examples}
In this section we begin with  some examples which demonstrate key phenomena in the qualitative setting. This motivates our methods and setup and in particular how they deviate from those in \cite{GT10}.  Then we work out the infinitary distribution of polynomial sequences on linear patterns. This includes most of the key ideas and phenomena that will appear in the quantitative setting (which is what we will ultimately use for our applications in Section \ref{s:app}), but is not muddied by quantitative bounds or the technical baggage introduced in their pursuit. This section will also be interspersed with examples. We hope that this Section is edifying in preparation for the more arduous task of reading Sections \ref{s:cl} and \ref{s:fact}. 

Recall that a filtration on a nilpotent simply-connected Lie group $G$ is a decreasing sequence of simply-connected Lie subgroups  $G=G_0 = G_1 \geq G_2 \geq \cdots$ with the property that $[G_i,G_j]\leq G_{i+j}$ for all $i,j$. A polynomial sequence $g(n) := \prod_i g_i^{\binom{n}{i}}$ on $G$ is said to be adapted to a filtered Lie group $G_\bullet$ if $g_i \in G_i$ for all $i$. As a general disclaimer, in this section we may use some standard terminology from this area without explicit introduction. We refer the reader to \cite{GT10} or \cite{GT12} for the development of the basic theory used here (on such an expedition, readers really need only acquaint themselves with basic qualitative definitions to make this rather informal section legible; we will provide our own definitions when doing things rigorously later on.)

\subsection{Non-flag examples}\label{ss:examples}

In the flag case, Green and Tao's counting lemma computes averages of the form $\E_{\vect x} \prod_{i=1}^t \phi(\psi_i(\vect x))$ where $\phi= f \circ g$ is a suitably `irrational' nilsequence in some filtered nilmanifold $G_\bullet$. Roughly speaking, if $G_\bullet$ is the `correct' filtered nilmanifold for the polynomial sequence $g$ (i.e. if $g$ cannot be written as a polynomial sequence on cosets of some smaller filtered subnilmanifold), then the distribution of $g^\Psi(\vect x) := (g(\psi_1(\vect x)), \ldots, g(\psi_t(\vect x)))$ in $G^t/\Gamma^t$ may be determined from the data $\{G_\bullet, \Psi\}$.  

\begin{definition}[The Leibman group]\label{d:leib-group}
Let $G_\bullet$ be a filtered nilmanifold and let $\Psi$ be a family of linear forms. Then define $G^\Psi:= \langle g_i^{v_i}, g_i \in G_i, v_i \in V^i \rangle$, where $V=V_\Psi$ is as defined in Subsection \ref{ss:flag}.
\end{definition}

The flag counting lemma says that if $g$ is irrational in $G_\bullet$ then $g^\Psi$ equidistributes in $G^\Psi / G^\Psi \cap \Gamma^t$, and so the average $\E_{\vect x} \prod_{i=1}^t \phi(\psi_i(\vect x))$ is asymptotically equal to the integral $\int_{G^\Psi/G^\Psi\cap \Gamma^t} f^{\otimes t}$. 

When moving to non-flag systems of forms, not only does $g^\Psi$ not not necessarily equidistribute on the Leibman group, but the notion of irrationality originally considered by Green and Tao does not capture enough information about the polynomial sequence $g$ to determine its distribution. The following example appears in \cite{GT20} and \cite{T20}.

\begin{example}\label{ex:flag-counter}
Let $\Psi:\Z^2 \to \Z^4$ be our non-flag system of linear forms from above:
\[\Psi(x,y) = (\psi_1(x,y),\psi_2(x,y),\psi_3(x,y),\psi_4(x,y))=(y, 2x+2y,x+3y,x).\]
Then, in $\R^4$,
\[V = \spa\{(0,2,1,1),(1,2,3,0)\} = \ker\{(3,0,-1,1), (2,-1,0,2)\}, \]
\[V^2 = \spa\{(0,4,1,1),(1,4,9,0),(0,4,3,0)\} = \ker\{(24,3,-4,-8)\},\]
and $V^3 = \R^4$.

Let $\alpha, \beta\in \R$ be such that $\{1,\alpha, \beta\}$ satisfy no nontrivial linear relations over $\Q$. Let $f: \R^2/\Z^2 \to \C$ be smooth and let $g(n) = (\alpha n, \beta n^2)$ determine a polynomial sequence on $\R^2$. Then (suppressing in our notation the quotient $\R^2 \to \R^2/\Z^2$) we have that  $\phi(n):=f(\alpha n, \beta n^2)$ is an irrational nilsequence in the filtered Lie group $\R^2$ with filtration $G_1 = \R^2$, $G_2 = 0 \times \R$, essentially because there is no filtered subnilmanifold with respect to which $g$ is a polynomial sequence. We would like our counting lemma to determine the average
\begin{equation}\label{e:ex-avg}
\E_{x,y} (f\circ g)(\psi_1(x,y))(f\circ g)(\psi_2(x,y))(f\circ g)(\psi_3(x,y))(f\circ g)(\psi_4(x,y)).
\end{equation}
To do so we find where $g^\Psi(x,y):=(g(\psi_1(x,y)),\ldots,g(\psi_4(x,y)))$ distributes within in $(\R^2/\Z^2)^4$ whereupon the average (\ref{e:ex-avg}) may be computed by the relatively easy task of integrating over a subtorus. 

Here the Leibman group is \[\{(a,  b):  a \in V,  b \in V+V^2\},\] 
suitably interpreted as a subgroup of $(\R^2)^4$. On the other hand, one may compute that in fact $g^\Psi$ is constrained to the following subgroup of $G^\Psi$: 
\[\{(a, b): a \in V,  b \in V^2\}. \]
Since $V^2 \subsetneq V+V^2$ we have that $g^\Psi$ takes values in a strict subgroup of the Leibman group and so will certainly not equidistribute there. (Note on the other hand that if the flag condition was satisfied by the system of linear forms, i.e. $V \leq V^2$ so $V^2 = V+V^2$, then the subgroup is equal to the Leibman group and indeed Green and Tao's counting lemma holds.)

Now define $\tilde g(n) = (\alpha n , \gamma n + \beta n^2)$, where $\{1,\alpha,\beta,\gamma\}$ is a $\Q$-linearly independent set. Then $\tilde g$ is irrational in the same filtration as $g$ and in this instance $\tilde g^\Psi$ is not constrained to the same subgroup of $G^\Psi$ as $g^\Psi$ was. In fact, $\tilde g^\Psi$ equidistributes in $G^\Psi$. The upshot is that in the non-flag setting, the irrationality of a polynomial sequence in $G_\bullet$ is not enough to determine the distribution of $g^\Psi$. 
\end{example}

In light of this example, one might conjecture in general that  the subgroup of $G^t$ on which $g^\Psi$ equidistributes may be expressed in terms of the sequence of subgroups on which the distinct degree parts of $g$ equidistribute. In our example with $g(n) = (\alpha,0)n + (0,\beta)n^2$ and $\tilde g(n) = (\alpha, \gamma)n + (0,\beta)n^2$, we have $(\alpha,0)n$ equidistributing in $\R\times 0$ whereas $(\alpha,\gamma)n$ equidistributes in $\R^2$ and indeed this precisely accounts for how the distribution of $g^\Psi$ differs from that of $\tilde g^\Psi$. It turns out that when the nilmanifold in question is abelian, one can indeed compute the distribution of $g^\Psi$ in $G^t/\Gamma^t$ if the  distribution of the coefficients of $g$ are known. However, this is not the case in the nonabelian setting, as the following example of linear sequences on the Heisenberg group demonstrates.

\begin{example}\label{ex:heis-counter}

Let $G$ be the Heisenberg group with filtration determined by $G_2 = [G,G]$. Let $\Psi, \psi_i, V$ be as above. Let $v_1 = (1,2,3,0)$ and $v_2=(0,2,1,1)$ so $V = \spa(v_1,v_2)$. Let 
\[ g_1 = \begin{pmatrix}
1 & \alpha & \gamma \\
0 & 1 & \beta \\
0 & 0 & 1 \\
\end{pmatrix}, \quad \tilde g_1 = \begin{pmatrix}
1 & \alpha & \frac{1}{2}\alpha \beta \\
0 & 1 & \beta \\
0 & 0 & 1 \\
\end{pmatrix},\]
where $1,\alpha, \beta, \alpha\beta, \gamma$ satisfy no nontrivial linear relations over $\Q$. Then both $g(n)=g_1^n$ and $\tilde g(n) = \tilde g_1^n$ equidistribute in $G/\Gamma$, but it transpires that $g^\Psi$  distributes differently to $\tilde g^\Psi$ in $G^4/\Gamma^4$. (Furthermore, both polynomial sequences are irrational by Definition A.6 of \cite{GT10}).

Both $g^\Psi$ and $\tilde g^\Psi$ take values in the Leibman group $G^\Psi$.  Define $\eta$ on $G^\Psi \leq G^4$ by \[\eta(h_1,h_2,h_3,h_4) = w \cdot (h_{1,2}, h_{2,2}, h_{3,2}, h_{4,2}),\] where \[h_i = \begin{pmatrix}
1 & h_{i,0} & h_{i,2} \\
0 & 1 & h_{i,1} \\
0 & 0 & 1
\end{pmatrix} \in G \]
for each $i$ and $w = (24,3,-4,-8) \in {V^2}^\perp$. An easy but mundane computation (using crucially that $w \in {V^2}^\perp$) reveals that in fact $\eta$ is a character on $G^\Psi$. 
We compute (with slight abuse of notation where we use a single matrix with vector-valued entries to denote a 4-tuple of matrices) 
\[g^\Psi(x,y) = g_1^{\Psi(x,y)} = g_1^{x v_1 + yv_2} =  \begin{pmatrix}
1 & \alpha(xv_1 + yv_2) & \gamma(xv_1 + yv_2) + \binom{xv_1 + yv_2}{2}\alpha \beta \\
0 & 1 &\beta(xv_1 + yv_2) \\
0 & 0 & 1
\end{pmatrix} \]
and so 
\begin{align*}
\eta(g^\Psi(x,y)) &= w\cdot v_1(\gamma - \frac{1}{2}\alpha \beta)x + w \cdot v_2 (\gamma - \frac{1}{2} \alpha \beta)y + \\* &\qquad + w\cdot v_1^2(\frac{1}{2}\alpha \beta)x^2 + 2 w \cdot v_1v_2 (\frac{1}{2}\alpha \beta) xy + w \cdot v_2^2 (\frac{1}{2}\alpha \beta)y^2.
\end{align*}
Since $w \in V_2^\perp$, we have that $\eta(g^\Psi(x)) = w\cdot v_1(\gamma - \frac{1}{2}\alpha \beta)x + w \cdot v_2 (\gamma - \frac{1}{2} \alpha \beta)y$, where $w\cdot v_1, w\cdot v_2 \ne 0$. Of course, completing the same computation with $\tilde g$ yields the same result with $\frac{1}{2}\alpha \beta$ in place of $\gamma$ and so in particular $\eta(\tilde g^\Psi(x)) = 0$. That is, $\tilde g^\Psi$ is constrained to the subgroup of $G^\Psi$ given by $\ker \eta$ whereas $g^\Psi$ is not. (As an application of our qualitative counting lemma Theorem \ref{t:main-qual}, we will see in Example \ref{ex:equid-proofs} that $g^\Psi$ in fact equidistributes on $G^\Psi / G^\Psi \cap \Gamma^4$ and $\tilde g^\Psi$ equidistributes on $\ker \eta / \ker \eta \cap \Gamma^4$.)
\end{example}

\subsection{Moving to the Lie algebra}\label{ss:lie-alg}

Let us now interpret the previous example in the Lie algebra $\g$ of the Heisenberg group $G$. Recall that $\log:G\to \g$ is a bijection given by
\[\log \begin{pmatrix}
1 & a & c \\
0 & 1 & b \\
0 & 0 & 1 
\end{pmatrix} = \begin{pmatrix}
0 & a & c-\frac{1}{2}ab \\
0 & 0 & b \\
0 & 0 & 0 
\end{pmatrix}.\] For convenience and brevity we will use the coordinates $(a,b,c-\frac{1}{2}ab)$ to denote the latter matrix in the previous sentence. The Lie bracket in these coordinates is 
\begin{equation}\label{e:heis-bracket}
[(a,b,c),(x,y,z)] = (0,0,ay-bx).
\end{equation} 
The algebra $\g$ inherits a filtration $\g_0=\g_1=\g$ and $\g_2 = [\g,\g] = \{(0,0,\R)\}$. In these coordinates, we have that $p(n) := \log g(n) = (\alpha,\beta,\gamma - \frac{1}{2}\alpha\beta)n$ and $\tilde p(n) := \log \tilde g(n) = (\alpha,\beta,0)n$.

Let us digress slightly and consider a general polynomial sequence $p$ on some Lie algebra $\g$ with linear forms $\Psi$ whereupon we are interested in determining where $p^\Psi$ equidistributes within $\g^t$. 

\begin{definition}[Filtration in the Lie algebra]\label{d:filtration}
A \textit{filtration} on $\g$ of \textit{step} at most $s$ a sequence $(\g_i)_{i=1}^s$ of Lie subalgebras of $\g$ such that $\g_0 = \g_1 \geq \g_2 \geq \cdots \geq \g_s$ and such that for all $i,j$, $[\g_i,\g_j]\leq \g_{i+j}$ where we define $\g_{s+1} = \g_{s+2} = \cdots = 0$.
\end{definition}

\begin{definition}[Polynomial sequence adapted to a filtration]\label{d:adapted}
A polynomial sequence defined by $p(n):=\sum_{i=1}^s a_in^i$ in $\g$ is \textit{adapted to $\g_\bullet$} if $a_i \in \g_i$ for all $i$. 
\end{definition}

It will prove convenient to deal with $\g \otimes \R^t$ rather than $\g^t$, where the map from the former to the latter on elementary tensors is  given by $\varphi: a \otimes v \mapsto (v_1a,v_2a,\ldots,v_ta)$.  This gives a vector space isomorphism and we may obtain a Lie algebra isomorphism by endowing $\g \otimes \R^t$ with a Lie bracket as follows:
\begin{align}\label{e:lie-bracket}
[a\otimes v,b\otimes w] &:= \varphi^{-1}([(v_1a,\ldots,v_ta),(w_1b,\ldots,w_tb)]) \non \\*
 &= \varphi^{-1}(([v_1a,w_1b],\ldots,[v_ta,w_tb])) \non \\*
 &= \varphi^{-1}((v_1w_1[a,b],\ldots,v_tw_t[a,b])) \non \\*
 &= [a,b]\otimes vw
\end{align}
where multiplication in the second entry is conducted coordinatewise. With this setup, and henceforth suppressing our $\varphi$ notation and implicitly identifying $\g^t$ with $\g\otimes \R^t$, if $p(n)=\sum a_in^i$ then $p^\Psi(\vect{x}):=(p(\psi_1(\vect{x})),\ldots, p(\psi_t(\vect{x}))) = \sum a_i\otimes \Psi^i(\vect{x}) \in \g \otimes \R^t$. 

Returning to our example,  $\tilde p^\Psi$ is contained in the subspace $(\R,\R,0)\otimes V$ at all inputs. A straightfoward computation using the Lie bracket defined above shows that  $(\R,\R,0)\otimes V + \g_2 \otimes V^2$ is the smallest Lie algebra containing the subspace $(\R,\R,0)\otimes V$ and indeed we will see later as an application of our qualitative counting lemma Theorem \ref{t:main-qual} in Example \ref{ex:equid-proofs} that $\tilde p^\Psi$ equidistributes in this subalgebra (modulo an appropriate lattice).  Importantly, $(\R,\R,0)\otimes V + \g_2 \otimes V^2$  is a proper subalgebra of  $\g \otimes V + \g_2 \otimes V^2 = \log G^\Psi$ and indeed $p^\Psi$ equidistributes in $\log G^\Psi$ (again, modulo an appropriate lattice).

\subsection{Notions of irrationality}
In this subsection we will restrict our attention to polynomial sequences $p$ with $p(0) = 0 \in \g$. In particular, notions of irrationality will only be defined for polynomial sequences satisfying this property. 

We begin by introducing some definitions, in particular that of a \textit{rational structure} on $\g$. A \textit{rational structure} on $\g$ is a rational Lie subalgebra  $\g_\Q$ of $\g$ such that the vector space isomorphism $\g \cong \g_\Q \otimes \R$ holds. Then, let $\g_\Q^\ast\subset \g^\ast$ be the linear maps on $\g$ which map $\g_\Q$ to $\Q$. In the coordinates in our Heisenberg example above, we have implicitly used the rational structure $\Q^3 \leq \R^3$. In what follows, let $\Hom_\Q(\g,\R)$ denote the Lie algebra homomorphisms from $\g$ to $\R$ which map $\g_\Q$ to $\Q$. Finally, many statements in this document are made with respect to a rational structure $\g_\Q$ on $\g$; at times  the underlying rational structure may not be explicitly referenced.

\begin{definition}[Qualitative additive irrationality]\label{d:additive-irrat}
A polynomial sequence $p$ in $\g$ is \textit{additively irrational} in $\g$ with respect to $\g_\Q$ if for all nontrivial $\eta \in \Hom_\Q(\g,\R)$ we have that $(\eta \circ p)(\vect x) \not \in \Q$ for some $\vect x$.
\end{definition}

In the following we invoke Lie's third theorem to glean the simply-connected Lie group $G$ from the Lie algebra $\g$. 

\begin{definition}[Qualitative equidistribution with respect to a rational structure]
A polynomial sequence $p$ in $\g$ \textit{equidistributes} with respect to a rational structure $\g_\Q$ if $\exp \circ p$ equidistributes in $G/\Gamma$ (as per Definition \ref{d:nilman-equid}), where $\Gamma$ is \textit{any} lattice in $G$ such that $\spa_\Q \log \Gamma = \g_\Q$. 
\end{definition}

It is not at all \textit{a priori} clear that the above is well defined (in particular that such a lattice exists and furthermore that the property of equidistribution is independent of the choice of such lattice). This is addressed in Appendix \ref{ss:lie-alg-equid}.

The following theorem of Leibman is our main tool for determining qualitative equidistribution. It says that the above two definitions are equivalent.

\begin{thm}[Leibman's criterion for qualitative equidistribution in the Lie algebra]\label{t:leibman-algebra}
Let $D$ be a positive integer, let $\g$ be a real, finite dimensional, nilpotent Lie algebra and let $\g_\Q$ be a rational structure on $\g$. Let $(p(\vect n))_{\vect n \in \N^D}$ be a polynomial sequence on $\g$. Then $(p(\vect n))_{{\vect n \in \Z^D}}$ equidistributes in $(\g,\g_\Q)$ if and only if $p$ is additively irrational in $(\g,\g_\Q)$. 
\end{thm}

We deduce the above statement from Leibman's original theorem (\cite{Lei05b}) in Theorem \ref{t:pr-leibman-algebra}.

For appropriately (i.e. \textit{linearly}) irrational $p$, it transpires that $p^\Psi$ is additively irrational in an explicit subalgebra of $\g^t$, whereupon we have determined its distribution by Theorem \ref{t:leibman-algebra}. Before defining linear irrationality, we will transfer Green and Tao's notion of irrationality \cite[Definition A.6]{GT10} to the Lie algebra. For the purposes of our current discussion, we give the following qualitative translation in which one removes the complexity requirement on characters and takes the limit $N\to \infty$.

\begin{definition}[Qualitative Green-Tao irrationality]\label{d:qual-gti}
An element $g\in G_i$ is $i$-irrational if for all  nontrivial characters $\eta:G_i \to \R$ which vanish on $G_{i+1}$ and $[G_j,G_{i-j}]$ for all $j$, and which map the lattice $\Gamma$ to $\Q$, we have $\eta(g)\not \in \Q$.
\end{definition} 

(A direct translation works with characters that map $\Gamma$ to $\Z$, but the above is equivalent and more natural in the qualitative setting.) By basic Lie theory one may interpret this in the Lie algebra with filtration $\g_\bullet := (\log G_i)_{i=1}^s$.  

\begin{definition}\label{d:hi}
Let $\g$ have filtration $\g_\bullet = (\g_i)_{i=0}^s$. For $i=1 \ldots, s$, define the subalgebra $\h_i$ to be the smallest Lie algebra containing $\g_{i+1}$ and $[\g_j,\g_{i-j}]$ for all $j$.
\end{definition}

The following definition is what one obtains when one lifts qualitative Green-Tao irrationality to the Lie algebra. We will call it `filtration irrationality'. 

\begin{definition}[Qualitative filtration irrationality]\label{d:qual-fi}
We say $a\in \g_i$ is \textit{$i$-filtration irrational} if for all $l\in \g_\Q^\ast$ which are nontrivial on $\g_i$ and vanish on $\h_i$, we have $l(a)\not \in \Q$. A polynomial sequence $p(n) = \sum_{i=1}^s a_in^i$ is \textit{filtration irrational} in $\g_\bullet$ if $a_i$ is $i$-filtration irrational for $i=1,\ldots,s$.
\end{definition} 

It is standard (\cite[Theorem 5.1.8]{CG90}) that for any lattice $\Gamma \leq G$, $\spa_\Q \log \Gamma$ is a rational structure in $\g$. With the rational structure $\g_\Q$ on $\g$ chosen in this way, one may check that every $l$ as in Definition \ref{d:qual-fi} induces a character $\eta$ on $G_i$ with the properties in Definition \ref{d:qual-gti} (and conversely every $\eta$ induces an $l\in \g_\Q^\ast$ which vanishes on $\h_i$).

In the examples in the previous subsection, we were somewhat vague about the notion irrationality. One is now in a position to verify that all polynomial sequences that feature in Examples \ref{ex:flag-counter} and \ref{ex:heis-counter} are filtration irrational and that ultimately filtration irrationality is insufficient to determine the distribution of $p^\Psi$ in the non-flag setting. 

\begin{definition}[Rational subspace]\label{d:rat-subspace}
A subspace $S$ of $\g$ is \textit{rational} with respect to a rational structure $\g_\Q$ if it may be written as the kernel of a set of elements of $\g_\Q^\ast$.
\end{definition}

\begin{definition}[Qualitative linear irrationality]\label{d:qual-li} Let $S$ be a rational subspace of $\g$. An element $a \in S$ is \textit{linearly irrational} in $S$ if for all $l\in \g_\Q^*$ which are nontrivial on $S$, we have $l(a)\not \in \Q$. A polynomial sequence $p(n):=\sum_{i=1}^s a_in^i$ is \textit{linearly irrational} in a sequence of subspaces $S_\bullet=(S_i)_{i=1}^s$ if $a_i$ is linearly irrational in $S_i$ for $i=1,\ldots,s$.
\end{definition}

The notion of linear irrationality does not refer to the Lie algebra structure on $\g$ and so may be defined analogously for vector spaces with fixed rational structures. We note that linear irrationality is robust under rational (vector space) automorphisms.

In our Heisenberg example above, recalling that $\alpha$, $\beta$, $\gamma$ and $1$ are linearly independent over $\Q$, we have that $\tilde p(n) =(\alpha,\beta,0)n$ is linearly irrational in $S_\bullet = (S_1:=\R \times \R \times 0,S_2 := 0)$, but is not linearly irrational in $(\R^3,0)$, or in $(\R\times \R \times 0 , 0\times 0\times \R)$. Define 
\[  p'(n) = (\alpha, \beta, 2\alpha+\beta)n + (0,0,\gamma)n^2, \qquad p''(n) = (\alpha, \beta, 2\alpha + \beta+1/3)n + (0,0,\gamma)n^2.\]
Then $p'$ is linearly irrational in $(\ker(2,1,-1),0 \times 0 \times \R)$. On the other hand, $p''$ is not linearly irrational in any sequence of subspaces of $\g$ because there is $l \in \g_\Q^\ast$ with $0 \ne l(a_1)\in \Q$.

In general, any polynomial sequence which has a coefficient $a_i$ with $0\ne l(a_i) \in \Q$ for some $l \in \g_\Q^\ast$ is not linearly irrational in any sequence of subspaces. If, on the other hand, all coefficients $\{a_i\}_i$ of $p$ satisfy the property that $l(a_i)\in \Q$ implies $l(a_i)=0$, then $p$ is linearly irrational in the sequence of  rational subspaces $S_\bullet=(S_i)_{i=1}^s$ defined by $S_i = (a_i^\perp \cap \g_\Q^\ast)^\perp$. 

Linear irrationality is a stronger notion than filtration irrationality in the sense that if $p$ is linearly irrational with respect to some $S_\bullet$ then there exists some $\g'_\bullet$ (potentially on a subalgebra $\g'\leq \g$) with respect to which $p$ is filtration irrational (see Lemma \ref{l:li-implies-fi}).
 The converse is not true, as is demonstrated by the example of $p''$ above which is filtration irrational with respect to the lower central series filtration. Finally, also note that if $p$ is adapted to a rational filtration $\g_\bullet$ and is linearly irrational with respect to $S_\bullet$ then we automatically have that $S_i \leq \g_i$ for all $i$ and so $[S_i,S_j] \subset \g_{i+j}$ for all $i,j$.

Unfortunately the notions of linear and filtration irrationality are not quite as simple in the quantitative setting. For example, a polynomial sequence may be quantitatively filtration/linearly irrational in two distinct sequences of subspaces -- see Example \ref{ex:irrat}. The relationship between the two notions of irrationality is also not so clear cut in the quantitative setting (though they are related -- see Lemma \ref{l:Ti+hi=gi}). For this reason we will explicitly keep track of the sequence of subspaces with respect to which we are asserting quantitative linear irrationality and independently the filtered Lie algebra with respect to which we are asserting quantitative filtration irrationality.

\subsection{The counting lemma}
We can now prove a qualitative counting lemma for polynomial sequences on linear patterns. Again we will restrict our attention to polynomial sequences with $p(0) = 0\in \g$; in the general case one simply obtains equidistribution on a coset of the algebra studied in this section. Taking for granted our algebraic preliminaries, Theorem \ref{t:main-qual} is rather straightfoward.

\begin{definition}\label{d:gPsiBullet}
Given a system of $t$ linear forms $\Psi$ and a sequence of rational subspaces $S_\bullet= (S_i)_{i=1}^s$ of $\g_\bullet$, define the Lie algebra $g^\Psi(S_\bullet)$ as the smallest Lie subalgebra of $\g^t$ which contains $S_i \otimes V^i$ for all $i$ (where recall $V=V_\Psi$ is defined in Subsection \ref{ss:flag}).
\end{definition}

The qualitative counting lemma may be stated informally as follows: if  a polynomial sequence $p$ is linearly irrational in $S_\bullet $ then $p^\Psi$ equidistributes $\g^\Psi(S_\bullet)$. We will offer a more explicit description of $g^\Psi(S_\bullet)$ in Section \ref{s:gPsiSbullet}.

\begin{lemma}\label{l:rat-poly-rat-coeffs}
Let $p: \R^D \to \R$ be a polynomial such that $p(\Z^D)\subset \Q$. Then the coefficients of $p$ (with respect to the monomial basis) lie in $\Q$.
\end{lemma}
\begin{proof}
For univariate polynomials ($D=1$), the coefficients may be recovered by evaluating the polynomial at sufficiently distinct points and basic linear algebra over $\Q$.  The multivariate case follows from an easy induction on the number of variables.
\end{proof} 

We are ready to prove the qualitative counting lemma.

\begin{thm}[Qualitative counting lemma]\label{t:main-qual}
Suppose that $p$ is a linearly irrational polynomial sequence in $S_\bullet$. Then the polynomial sequence $p^\Psi(\vect{x})$ equidistributes on $\g^\Psi(S_\bullet)$ with respect to the rational structure $\g^\Psi(S_\bullet) \cap \g_\Q^t$.  
\end{thm}
\begin{proof}
We use Theorem \ref{t:leibman-algebra} and so will show that $p^\Psi$ is additively irrational in $\g^\Psi(S_\bullet)$. Let $p(n) := \sum_{i=1}^s a_in^i$. Let $\eta \in \Hom_\Q(\g^\Psi(S_\bullet), \R)$.\footnote{Recall that we use the notation $\Hom_\Q$ to refer to  homomorphisms of Lie algebras which map $\g_\Q$ to $\Q$.} In $\g^\Psi(S_\bullet) \leq \g\otimes \R^t$, we have that 
\begin{equation}\label{e:eta-g-psi}
\eta(g^\Psi(\vect{x})) = \sum_{i=1}^s \eta(a_i\otimes \Psi(\vect x)^i)  = \sum_{m \in \cM}\sum_{i=1}^s \eta(a_i\otimes v_{m,i})m(\vect{x}),
\end{equation}
where the $\{v_{m,i}\}_{m\in \cM}$ span $V^i$ as in Subsection \ref{ss:flag}. 
By Lemma \ref{l:rat-poly-rat-coeffs}, if $\eta(g^\Psi(\vect x)) \in \Q$ for all $\vect x \in \Z^D$ then each $\eta(a_i \otimes v_{m,i}) \in \Q$ . But each $\eta(\cdot \otimes v_{m,i})$ lies in $\g_\Q^\ast$ and so by the linear irrationality of $a_i$ in $S_i$ we have that in fact $\eta(a_i \otimes v_{m,i}) = 0$. The kernel of a rational homomorphism is a rational subspace and so $\eta(S_i \otimes v_{m,i})=0$, since $S_i$ is necessarily the smallest rational subspace containing $a_i$. Next, the $\{v_{m,i}\}_{m \in \cM}$ span $V^i$, so we have $\eta(S_i \otimes V^i) = 0$ for all $i$. Furthermore,  since $\eta$ is a Lie algebra homomorphism to an abelian Lie algebra, $\eta$ vanishes on the smallest Lie algebra containing the subspaces $S_i \otimes V^i$ and so $\eta$ is trivial on $\g^\Psi(S_\bullet)$. This completes the proof.
\end{proof}

\begin{remark}
In the flag setting, if $\eta$ is a character on $G^t$ (or a subgroup thereof), the `pullback' $\eta (\cdot^{v_{m,i}} )$ is an $i$-horizontal character  on $G_i$ (\cite[Definition A.5]{GT10}). However, in the non-flag setting, $\eta(\cdot^{v_{m,i}})$ may fail to possess obviously discernible structure as a map on $G_i$ (cf. Example \ref{ex:heis-counter} and the character $\eta$ from there). The key technical point in the above is that characters on $\g^t$ (or subalgebras thereof) pull back to linear maps on $\g$. That is, $\eta(\cdot \otimes v_{m,i})$ is a rational linear map on $\g$ whence linear irrationality is sufficient to determine the distribution of $p^\Psi$.
\end{remark}

Note that Theorem \ref{t:main-qual} yields equidistribution statements for polynomial sequences which are not filtration irrational.
 
\begin{example}
 In the Heisenberg group with usual coordinates, let $p(n) = (\alpha,\alpha,0)n$. Then $p$ is linearly irrational in $S_\bullet:=(S_1:=\{(x,x,0):x\in \R\}, S_2:=0)$. Furthermore, we have from Equation (\ref{e:heis-bracket}) that $[S_1,S_1] = 0$ (as is expected since this example is really coming from an abelian group). Theorem \ref{t:main-qual} says that $p^\Psi$ equidistributes on $S_1 \otimes V$.
\end{example}

We are also able to prove the equidistribution claims we made in Examples \ref{ex:flag-counter} and \ref{ex:heis-counter}, and in Subsection \ref{ss:lie-alg}. 

\begin{example}[Equidistribution for Examples \ref{ex:flag-counter} and \ref{ex:heis-counter}]\label{ex:equid-proofs}
In the abelian setting we may identify the Lie group with the Lie algebra. Let $p(n) = (\alpha, 0)n + (0,\beta)n^2$ on $\R^2$ be the first polynomial sequence from Example \ref{ex:flag-counter}. Theorem \ref{t:main-qual} yields that $p^\Psi$ equidistributes on $(\R\times 0)\otimes V + (0\times \R)\otimes V^2$. On the other hand, the second polynomial sequence from Example \ref{ex:flag-counter} $\tilde p(n) = (\alpha, \gamma)n + (0,\beta)n^2$ equidistributes on $\R^2 \otimes V + (0\times \R)\otimes V^2$.

Moving to Example \ref{ex:heis-counter}, we saw in Subsection \ref{ss:lie-alg} that the relevant polynomial sequences may be transferred to the Lie algebra as $p(n) = (\alpha, \beta, \gamma-\frac{1}{2}\alpha \beta)n$ and $\tilde p(n) = (\alpha, \beta, 0)n$. Theorem \ref{t:main-qual} says that $p^\Psi$ equidistributes in  the smallest Lie algebra containing $\R^3 \otimes V$, which is $\R^3 \otimes V + (0\times 0\times \R) \otimes V^2=\g \otimes V + \g_2 \otimes V^2$. We also obtain that $\tilde p^\Psi$ equidistributes in $(\R\times \R \times 0)\otimes V + \g_2 \otimes V^2$. One may furthermore check that this subalgebra maps to the subgroup $\ker \eta$ from Example \ref{ex:heis-counter} under the exponential map. 
\end{example}

In each of the examples above, computing explicitly the Lie algebra $\g^\Psi(S_\bullet)$ is not a difficult task. Of course, this may not be the case in general. We will address this in Section \ref{s:gPsiSbullet}.

\subsection{Factorisation}\label{ss:qual-fact}
We have demonstrated how to obtain an equidistribution result for polynomial sequences which satisfy a special property: linear irrationality. Ultimately, one wants to understand the equidistribution of arbitrary polynomial sequences (we will again focus on $p$ with $p(0)=0$). The purpose of this section is to show (Proposition \ref{p:qual-factorise}) that one may write an arbitrary $p$ as a product $p = p'\ast r$ where $p'$ is linearly irrational and $r$ is rational (and as before, $\ast$ is group multiplication conducted in the Lie algebra). Then, roughly speaking, one may partition $\Z^D$ into appropriately-spaced subprogressions where we understand the equidistribution of $p$ on each one. 

Part of the power of Green and Tao's counting lemma is that it requires a surprisingly small amount of data from a polynomial sequence $g$ to determine the distribution of $g^\Psi$ in the flag case. In the non-flag case, we require more data to determine distribution and so the task of reducing a general nilsequence to something that our counting lemma can handle  (i.e. factorising for linear irrationality) is inevitably a little more involved. This subsection may be viewed as a warm up for the quantitative factorisation theory in Section \ref{s:fact}.

Linear irrationality is a linear algebraic notion and indeed one may, by natural linear algebraic arguments, write an arbitrary $a \in \g$ as the sum of a linearly irrational element and a rational element. 

\begin{lemma}\label{l:qual-add-decomp}
Let $a \in \g$. Then there exists $a_p$ which is linearly irrational in some rational subspace $S \leq \g$ and $a_r \in \g_\Q$  such that $a= a_p+a_r$. 
\end{lemma}

One may apply this to each coefficient of an arbitrary polynomial sequence $p$ to obtain  an additive decomposition $p = p' + r$, where $p'$ is linearly irrational and $r$ is rational. This additive decomposition does not, however, immediately determine a factorisation with respect to group multiplication, which is ultimately what is required.

To later obtain our multiplicative factorisation, we will need a stronger additive decomposition, Lemma \ref{l:qual-relative-decomp}, of which Lemma \ref{l:qual-add-decomp} is a special case. First we will need the following definition of linear irrationality relative to a subspace. 

\begin{definition}[Relative linear irrationality]\label{d:qual-rel-li}
Let $U, T$ be rational subspaces of $\g$. An element $a \in U$ is \textit{linearly irrational in $U\bmod T$} if for all $l \in \g_\Q^\ast$ which are nontrivial on $U$ and which vanish on $U\cap T$ we have $l(a) \not \in \Q$.
\end{definition}

In what follows we use $T_\Q$ to denote the space $T \cap \g_\Q$, and similarly for other subspaces. If $T$ is rational with respect to $\g_\Q$, then $T_\Q$ is a rational vector space of full dimension (and indeed topologically dense) inside $T$.

\begin{lemma}[Qualitative additive decomposition for linear irrationality]\label{l:qual-relative-decomp}
Let $U,T$ be rational subspaces of $\g$ and let $a \in U$ be  linearly irrational in $U\bmod T$. Then there exist $a_p\in U$ and $a_r \in (U \cap T)_\Q$ such that $a_p$ is linearly irrational in some rational subspace $U'\leq U$ and such that $a= a_p+a_r$. 
\end{lemma}
\begin{proof}

Let $L\leq \g_\Q^\ast$ be the subspace of elements which are nontrivial on $U$ and such that $l(a) \in \Q$.  Choose a basis $\cB$ for $L$. The elements of $\cB$ remain linearly independent when restricted to $U \cap T$ (if not, choose a nontrivial linear combination of them which is the trivial map on $U \cap T$ to obtain a contradiction with the linear irrationality of $a$ in $U\bmod T$). Furthermore, since $U,T$ are rational subspaces, so too is $U\cap T$ and so $\cB$ is linearly independent when restricted to a set of maps on the $\Q$-vector space $(U \cap T)_\Q$. Thus there exists $a_r\in (U \cap T)_\Q$ such that $l(a) =l(a_r)$ for all $l \in \cB$, and therefore for all $l \in L$. Note that since $a_r \in U_\Q$, the subspace of elements which are nontrivial on $U$ and such that $l(a-a_r) \in \Q$ is also equal to $L$. Ultimately $l(a-a_r)=0$ whenever $l(a-a_r)\in \Q$ for all $l \in \g_\Q^\ast$ (this is trivially true for maps which are trivial on $U$). Thus $a-a_r$ is linearly irrational in $\ker\{ l\in \g_\Q^\ast : l(a-a_r) \in \Q \} \leq U$.

\end{proof}

Note that the above lemma does not depend at all on the Lie bracket structure and so holds in vector spaces. 

\begin{definition}
Let $\g_\bullet$ be a filtration of step $s$ on $\g$ and let $U_\bullet, T_\bullet \leq \g_\bullet$ be two sequences of rational subspaces with $U_i,T_i \leq \g_i$ for all $i$. A polynomial sequence $p:= \sum_{i=1}^s a_i n^i$ in $\g$ is \textit{linearly irrational in} $U_\bullet\bmod T_\bullet$ if  $a_i$ is linearly irrational $U_i\bmod T_i$ for all $i$.
\end{definition}

\begin{definition}
A polynomial sequence $p(n):= \sum_{1=1}^s a_in^s$ adapted to $\g_\bullet$ is \textit{rational} if $a_i \in (\g_i)_\Q$ for all $i$.
\end{definition}

 In what follows, let $(p)_i$ denote the $i$th coefficient of a polynomial sequence $p$. Applying Lemma \ref{l:qual-relative-decomp} to each coefficient of a polynomial sequence we obtain the following.

\begin{cor}\label{c:qual-poly-decomp}
Let $U_\bullet,T_\bullet$ be sequences of rational subspaces of $\g_\bullet$ and let $p$ be a polynomial sequence which is linearly irrational in $U_\bullet\bmod T_\bullet$. Then there are polynomials $p'$ and $r$ adapted to $U_\bullet$   and $(U_\bullet \cap T_\bullet)_{\Q}:= ((U_i\cap T_i)_\Q)_{i=1}^s$ respectively such that $p'$ is linearly irrational in some sequence of rational subspaces $U'_\bullet \leq U_\bullet$ and $p = p' + r$.
\end{cor} 

Now we pursue a factorisation with respect to the $\ast$ operation. We will need the fact that the set of polynomial maps adapted to $\g_\bullet$ form a group. This result (in the Lie group, and in greater generality) is originally due to Leibman \cite{Lei02}. See also \cite[Proposition 6.2]{GT12} for a related statement and different proof. The result that we need in the Lie algebra is more obvious and follows from Baker-Campbell-Hausdorff. 

\begin{lemma}\label{l:polys-group}
The set of polynomials adapted to $\g_\bullet$ forms a group under the $\ast$ operation.
\end{lemma}

\begin{lemma}\label{l:qual-prod-rational}
The product of two rational polynomial sequences is rational.
\end{lemma}
\begin{proof}
This follows immediately from the Baker-Campbell-Hausdorff formula $x \ast y = x + y + \frac{1}{2}[x,y] + \ldots$ and the fact the rational subspaces $({\g_i})_\Q$ are rational subalgebras which obey the same filtration as $\g_\bullet$.
\end{proof}

\begin{prop}[Qualitative multiplicative factorisation for linear irrationality]\label{p:qual-factorise}
Let $p$ be a polynomial sequence in $\g$ adapted to $\g_\bullet$. We may write $p = p' \ast r$ where $p'$ is linearly irrational with respect to some sequence of rational subspaces $S'_\bullet \leq \g_\bullet$ and $r$ is a rational polynomial sequence, both adapted to $\g_\bullet$. 
\end{prop}
\begin{proof}
Invoke Corollary \ref{c:qual-poly-decomp} with $U_\bullet = \g_\bullet$ and $T_\bullet =(\g, \ldots, \g)$ to obtain $p_1$ which is linearly irrational in some sequence $S^{(1)}_\bullet$ and $r_1$ adapted to $(S_\bullet)_\Q$ such that $p = p_1 + r_1$. By Baker-Campbell-Hausdorff and the fact that $\g/\g_2$ is abelian we have that \[p = p_1 \ast r_1 \mod \g_2.\]  

Set $T^{(2)}_\bullet := (\g_2,\g_2,\ldots, \g_2)$. Then $p \ast r_1^{-1}$ is a polynomial sequence adapted to $\g_\bullet$ by Lemma \ref{l:polys-group}, and is linearly irrational in $(S_\bullet^{(1)} + T^{(2)}_\bullet) \bmod  T^{(2)}_\bullet$.\footnote{Here we recall that all elements of the filtration $\g_\bullet$  are rational in $\g$, so in particular $T_\bullet$ is a sequence of rational subspaces.} Next apply Corollary \ref{c:qual-poly-decomp} with $U_\bullet = S^{(1)}_\bullet + T^{(2)}_\bullet$ and $T_\bullet = T^{(2)}_\bullet$ to $p\ast r_1^{-1}$ to find $p_2$ which is linearly irrational in some $S^{(2)}_\bullet$ and $r_2$ which is adapted to $(S^{(1)}_\bullet \cap T^{(2)}_\bullet)_\Q$ such that $p\ast r_1^{-1} = p_2 + r_2$. By Baker-Campbell-Hausdorff and the fact that $[\g,\g_2] \subset \g_3$ we then have $p\ast r_1^{-1} = p_2 \ast r_2 \bmod \g_3$ and so
\[p = p_2 \ast r_2' \mod \g_3, \]
where $r_2':= r_2 \ast r_1$ is rational by Lemma \ref{l:qual-prod-rational}. Continue in this way to eventually find $p':=p_s$, $S'_\bullet := S^{(s)}_\bullet$ and $r:= r_s'$.
\end{proof}

\section{Description of $\g^\Psi(S_\bullet)$ and comparison with $\log G^\Psi$}\label{s:gPsiSbullet}

Our counting lemma Theorem \ref{t:main-quant} says that if $p$ is a polynomial sequence which is quantitatively linearly irrational in $S_\bullet$, then $p^\Psi$ quantitatively equidistributes in a Lie subalgebra $\g^\Psi(S_\bullet)$ of $\g^t$. We have proven a qualitative version of this result Theorem \ref{t:main-qual} which gives qualitative equidistribution on the same Lie algebra $\g^\Psi(S_\bullet)$. So far, however, we have only described $\g^\Psi(S_\bullet)$ as the smallest Lie subalgebra of $\g^t$ which contains $S_i\otimes V_\Psi^i$ for all $i$. The goal of this section is to give descriptions of $\g^\Psi(S_\bullet)$ which are more explicit and amenable to applications. 

Let $S_\bullet := (S_i)_{i=1}^s$ be a sequence of rational subspaces of $\g$ with respect to a fixed rational structure $\g_\Q$. Let $W_1 = S_1$ and iteratively define $W_i = \spa \{S_j, [W_j,W_{i-j}] \text{ for } j=1,\ldots, i-1\}$ for $i=2,\ldots, s$. That is, $W_i$ comprises all degree $i$ commutators in the $S_j$, where the degree of a commutator is the sum of the subscripts it contains. Note that if $\g_\bullet$ is a filtration such that $S_i \leq \g_i$ for all $i$ then $W_i \leq \g_i$ for all $i$.

\begin{prop}[General description of $\g^\Psi(S_\bullet)$]\label{p:gPsi-Sbullet}
Let $S_\bullet$ be a sequence of rational subspaces of $\g_\bullet$. Let $\Psi$ be a system of linear forms. Then $\g^\Psi(S_\bullet) = \sum_{i=1}^s W_i \otimes V_{\Psi}^i$.
\end{prop}
\begin{proof}
Recall from Equation (\ref{e:lie-bracket}) the Lie bracket operation on $\g\otimes \R^t$. It follows immediately that for all $i,j$ we have $[W_i \otimes V^i, W_j \otimes V^j] \subset W_{i+j} \otimes V^{i+j}$ (where $W_{s+1}=W_{s+2}=\cdots =0$) and so $\sum_{i=1}^s W_i \otimes V^i$ is indeed a Lie subalgebra of $\g^t$. It is also immediate from definitions that $\sum_{i=1}^s W_i \otimes V^i$ is generated by $\{S_i\otimes V^i\}_{i=1}^s$ under commutators and linear spans. 
\end{proof}

Recall that we use $G^\Psi$ to denote the Leibman group (Definition \ref{d:leib-group}); we will call $\log G^\Psi = \sum_{i=1}^s \g_i\otimes V_\Psi^i$ the \textit{Leibman algebra}. We now pursue Proposition \ref{p:comp-leib-gpsi}, a description of $\log G^\Psi$ which is more readily comparable to that of $\g^\Psi(S_\bullet)$ in the above proposition. In so doing we will make the assumption that $S_i + \h_i = \g_i$ for all $i$ (recall Definition \ref{d:hi}). The goal of the upcoming lemmas is to motivate why this is reasonable.

For a sequence of subspaces $S_\bullet \leq \g_\bullet$, we will say that a filtration $\g'_\bullet\leq \g_\bullet$ is \textit{minimal} for $S_\bullet$ if $S_\bullet \leq \g'_\bullet$ and there is no proper subfiltration $\g''_\bullet < \g'_\bullet$ with $S_i \leq \g''_i$ for all $i$.  

\begin{lemma}
For a sequence of subspaces $S_\bullet \leq \g_\bullet$, if $\g_\bullet$ is minimal for $S_\bullet$ then $S_i + \h_i = \g_i$ for all $i$.\footnote{The converse is also true though we do not need such a statement.}
\end{lemma}
\begin{proof}
If not, for the smallest $i$ such that $S_i + \h_i \ne \g_i$, set $\g_i' = S_i +\h_i$, and otherwise set $\g_j'=\g_j$ for all $j\ne i$. This yields a proper subfiltration $\g_\bullet' < \g_\bullet$. 
\end{proof}

\begin{lemma}\label{l:li-implies-fi}
Let $p$ be linearly irrational with respect to a sequence of subspaces $S_\bullet \leq \g_\bullet$. Then $p$ is filtration irrational with respect to the filtration $\g_\bullet'\leq \g_\bullet$ which is minimal for $S_\bullet$. 
\end{lemma}
\begin{proof}
This follows easily from definitions. We omit the details. 
\end{proof}

\begin{lemma}\label{l:si+hi-qual}
Let $p(n):=\sum_{i=1}^s a_i n^i$ be filtration irrational in $\g_\bullet$. Then for any sequence of rational subspaces $T_\bullet$ with $a_i \in T_i \leq \g_i$, we have $T_i + \h_i = \g_i$ for all 
$i$.

In particular, if $p$ is linearly irrational in $S_\bullet$ then $S_i + \h_i = \g_i$ for all $i$. 
\end{lemma}
\begin{proof}
That $T_i + \h_i \leq \g_i$ is immediate. For all $l\in \g_{i\Q}^\ast$ which vanish on $\h_i$ we have $l(a_i) \not \in \Q$. In particular, we have $\g_i = (a_i^\perp \cap \g_{i\Q}^\ast \cap \h_i^\perp)^\perp = (a_i^\perp \cap \g_{i\Q}^\ast)^\perp + \h_i$ where orthogonal complements are taken in $\g_i$ and $\g_i^\ast$.  But $(a_i^\perp \cap \g_{i\Q}^\ast)^\perp$ is the smallest rational subspace of $\g_{i}$ containing $a_i$. Thus $T_i + \h_i \geq \g_i$.
\end{proof}

The following is the main ingredient in the proof of Proposition \ref{p:comp-leib-gpsi}.

\begin{lemma}\label{l:maximal-Wi} 
Suppose that the sequence of subspaces $(S_i)_{i=1}^s \leq \g_\bullet$ satsifes $S_i + \h_i = \g_i$ for all $i$. Then the sequence of subspaces $(W_i)_{i=1}^s$ satisfies 
\[W_i + \g_{i+1} = \g_i ,\]
for all $i$.
\end{lemma}
\begin{proof}
The proof proceeds by induction and the base case is easy from definitions. Expanding the definition of $W_i$, we will show that
\begin{equation}\label{e:max-Wi}
S_i + \sum_{j=1}^{i-1}[W_j,W_{i-j}] + \g_{i+1} = \g_i.
\end{equation} 
Now for $j=1, \ldots, i-1$ we have $W_j \leq \g_j$ and so by the filtration $\g_\bullet$ we get \[[W_j + \g_{j+1},W_{i-j}+\g_{i+1-j}] + \g_{i+1} = [W_j,W_{i-j}]+\g_{i+1}.\] By induction, therefore, Equation (\ref{e:max-Wi}) is equivalent to
\[S_i + \sum_{j=1}^{i-1}[\g_j,\g_{i-j}] + \g_{i+1} = \g_i.\]
But the left hand side of the above is just $S_i + \h_i$.
\end{proof}

\begin{prop}[Comparison of Leibman algebra and $\g^\Psi(S_\bullet)$]\label{p:comp-leib-gpsi}
Suppose $S_\bullet$ is a sequence of rational subspaces of $\g_\bullet$ such that $S_i + \h_i = \g_i$. Then 
\[\sum_{i=1}^s \g_i \otimes V_\Psi^i = \sum_{i=1}^s W_i \otimes \left(\sum_{j=1}^i V_\Psi^j\right) = \g^\Psi(S_\bullet) + \sum_{i=1}^s W_i \otimes \left(\sum_{j=1}^{i-1} V_\Psi^j\right).\]
\end{prop}
\begin{proof}
From Lemma \ref{l:maximal-Wi} we have $W_i + \g_{i+1}=\g_i$ and so $\g_i = \sum_{j=i}^s W_j$. The result follows by switching the order of summation.
\end{proof}

\begin{cor}\label{c:gpsi-flag}
If $\Psi$ satisfies the flag property and $\g_\bullet$ is minimal for $S_\bullet$ then $\g^\Psi(S_\bullet)$ is equal to the Leibman algebra.
\end{cor}
\begin{proof}
If $\Psi$ satisfies the flag property then for all $i$ we have $\sum_{j=1}^i V_\Psi^j = V_\Psi^i$; invoke Proposition \ref{p:gPsi-Sbullet} and \ref{p:comp-leib-gpsi}.
\end{proof}

\section{The counting lemma}\label{s:cl}

Our first goal in the quantitative setting will be to obtain a counting lemma for polynomial sequences which are (quantitatively) linearly irrational. Throughout this section and the next, we will assume that the dimension $d$ of $\g$, the step $s$ of $\g$, the number of linear forms $t$ and the number of variables $D$ in these linear forms  are all of bounded size; in particular we will always allow constants to depend on these parameters without necessarily indicating this in our notation.

We will need to fix a basis for $\g$ in order to quantify various notions. (Indeed, in the absence of a choice of basis for $\g$ we do not even have a canonical metric or (Lebesgue) measure on $\g$.) 

\begin{definition}[Rational basis]\label{d:rat-basis} 
A basis $\cX$ for $\g$ is \textit{rational} if its structure constants lie in $\Q$.
\end{definition}

 The $\Q$-span of a rational basis is a rational structure on $\g$. In what follows, we fix a rational basis $\cX$ for $\g$ and the corresponding rational structure $\g_\Q := \spa_\Q(\cX)$. If statements whose meaning implicitly depends upon a basis are made without reference to one, it may be assumed that $\cX$ is the chosen basis. Furthermore, we will denote by $\langle \cX \rangle$ the abelian group generated by $\cX$, that is $\spa_\Z \cX$. We will also need to conduct quantitative arithmetic on subspaces of $\g$ and so for essentially the same reasons, we need to identify $\Z$-bases for sublattices obtained by intersecting $\langle \cX \rangle$ with a rational subspace $S$. This may be done canonically by invoking the Hermite normal form. For details and definitions see Appendix \ref{a:lin-alg}, but for the purposes of a more casual reading of this section, one may interpret \textit{Hermite basis} to mean `a canonically-chosen basis $\cX'=\cX'_{S,\cX}$ for $S$ such that $\cX'$ is also a $\Z$-basis for the lattice $S \cap \langle \cX \rangle$.' We caution that the action of taking a Hermite basis of a subspace is not transitive. In the case that $S = \g$, we have $\cX' = \cX$.

In the qualitative setting in Section \ref{s:examples}, if a polynomial sequence was linearly irrational with respect to any sequence of rational subspaces, then it was linearly irrational with respect to the sequence minimal rational subspaces of $\g$ which contained (respectively) the sequence of polynomial coefficients. 

We need to be a little more careful in the quantitative setting. Even in the single variable case (that is, for the trivial sequence of linear forms $\Psi:= (\psi_1(n))$ where $\psi_1(n) = n$) and for $A$ arbitrarily large,  $(A,N)$-irrationality (in the original Green-Tao sense \cite[Definition A.6]{GT10}) does not preclude a polynomial sequence from equidistributing on more than one nilmanifold of `complexity' $A^{O(1)}$.  

\begin{example}\label{ex:irrat}
Let $A$ be a large integer and let $N > 100A^3$. Let $p(n) := (\frac{1}{10A}, \frac{1}{10A-1})n$ define a polynomial sequence in the abelian Lie algebra $\g := \R^2$ with basis $\cX := \{(1,0),(0,1)\}$ and lattice $\langle \cX \rangle = \Z^2$. 

Firstly, let us first verify that in the limit $A \to \infty$, the quantitative distribution obtained from the Green-Tao counting lemma agrees with  our qualitative counting lemma Theorem \ref{t:main-qual}. Here, $a_1 = (\frac{1}{10A}, \frac{1}{10A-1})$ and the smallest rational subspace containing $a_1$ is $S_1:=\{(x,y) \in \R^2 : (10A-1)x = 10Ay\}$. Furthermore, $S_1$ is a subalgebra of $\R^2$ and has `complexity' $A^{O(1)}$ because it is defined by linear relations of height $A^{O(1)}$.\footnote{The complexity of a subspace is defined rigorously soon. See Definition \ref{d:comp-subspace}.}  We claim that $p$ is $o_{A \to \infty}(1)$ equidistributed in $S_1$. Clearly $S_1$ is isomorphic to $\R$, and $\varphi: (x,y) \mapsto x/(10A)$ is the isomorphism which maps $\cX' = \{(10A,10A-1)\}$ to $1 \in \R$. Thus we may analyse $p'(n) = \frac{n}{100A^2}$ in $\R/\Z$.\footnote{We transfer this example to $\R$ with lattice $\Z$ where it is more clear what the definition of the complexity of a character should be. One may check that transferring in this way is equivalent to our upcoming formal definition \ref{d:comp-map}.} Set $a_1':= \varphi(a_1) = \frac{1}{100A^2}$. For any nonzero $m \in \Z$ with $|m|\leq A$ we have $||ma_1'||_{\R/\Z} \geq \frac{1}{100A^2} > A/N$. Thus, by \cite[Definition A.6]{GT10}, $p'$ is $(A,N)$-irrational in $\R/\Z$, so $p$ is $(A,N)$-irrational in $S_1$ with respect to $\cX'$ and so invoking the single variable version of Green and Tao's counting lemma \cite[Lemma 3.7]{GT10} we have that $p$ is $o_{A\to \infty}(1)$ equidistributed in $S_1$, as is predicted by Theorem \ref{t:main-qual}. 

On the other hand, given any nonzero linear map $l: \g \to \R$ which maps $\cX$ to $\Z$ and has height at most $A$, we have by the coprimality of $10A$ and $10A-1$ that $||l(a_1)||_{\R/\Z} \geq \frac{1}{(10A)(10A-1)} >  \frac{A}{N}$, and so $p$ is $(A,N)$-irrational in $\g$. Therefore by \cite[Lemma 3.7]{GT10}, $p$ is $o_{A \to \infty}(1)$ equidistributed in all of $\R^2/\Z^2$.
\end{example} 

\begin{remark}
Green and Tao's quantitative factorisation of polynomial sequences \cite[Lemma 2.9]{GT10} states that given a polynomial sequence $g$ on a nilmanifold, \textit{at least one} of the following is true: (1) $g$ is $(A,N)$-irrational in the nilmanifold, or (2) one may find a  smooth sequence $\beta$ and a rational sequence $\gamma$ such that $g = \beta \ast g' \ast \gamma$, where $g'$ takes values in a complexity $O_A(1)$ subnilmanifold of strictly smaller total dimension. Example \ref{ex:irrat} exhibits a situation where these two eventualities are not mutually exclusive (note: the smooth and rational parts in our example are trivial). 
\end{remark}

\begin{definition}[Integer linear map] We call an element of $\g^\ast$ \textit{integer} (with respect to $\cX$) if it maps $\cX$ to $\Z$. We will denote by $\g^\ast_\cX$ the set of integer linear maps with respect to $\cX$  and we will denote by $\Hom_{\cX}(\g,\R)$ those elements of $\g^\ast_\cX$ which vanish on $[\g,\g]$ (i.e. are Lie algebra homomorphisms). 
\end{definition}

In the following we refer to the Hermite basis $\cX'_{S,\cX}$ whose formal definition is given in Definition \ref{d:hermite-basis}.

\begin{definition}[Complexity of an integer linear map]\label{d:comp-map}
The \textit{complexity} of an integer linear map $l: \g \to \R$ on a rational subspace $S \leq \g$ with respect to $\cX$ is the $\ell_\infty$ norm of $l$ with respect to the dual basis of $\cX'_{S,\cX}$.
\end{definition}

\begin{definition}[Complexity of a subspace]\label{d:comp-subspace}
Given a Lie algebra $\g$ with rational basis $\cX$ and a proper rational subspace $S \leq \g$, we say that $S$ has \textit{complexity} at most $C$ in $\g$ (with respect to $\cX$) if $S$ may be written as the kernel of a set of integer linear maps, each of which have complexity at most $C$ with respect to $\cX$. We will say that $\g$ has complexity $1$ in itself.
\end{definition}

\begin{definition}[Quantitative linear irrationality]\label{d:li}
For a Lie algebra $\g$ and a rational subspace $S \leq \g$ we will say that $a \in S$ is $(A,\eps)$\textit{-linearly irrational in $S$} with respect to $\cX$ if, for all integer linear maps $l:\g \to \R$ which are nontrivial on $S$ and have complexity at most $A$ on $S$,  we have  $||l(a)||_{\R / \Z} \geq \eps$.
\end{definition}

\begin{definition}[Quantitative linear irrationality for polynomial sequences]\label{d:lin-irrat-poly} 
For a Lie algebra $\g$ and a sequence of rational subspaces $S_\bullet := (S_i)_{i=1}^s$, $S_i \leq \g$, a polynomial sequence $(p(\vect n))_{\vect n \in [N]^D}$ in $\g$ is $(A,N)$\textit{-linearly irrational in $S_\bullet$} with respect to $\cX$ if for all $i$, all coefficients of all degree $i$ terms of $p$ are $(A,A/N^i)$-linearly irrational in $S_i$.
\end{definition}

We will record some easy consequences of these definitions for more convenient use later on. The proofs are just linear algebra and are omitted. We refer the reader to Appendix \ref{a:lin-alg} for some further details. Recall that throughout this section we are assuming that $\dim \g = O(1)$.

\begin{lemma}[Complexity of subspaces under operations]\label{l:comp-operations}
Let $C > 1$. Let $S$, $T$ be rational subspaces of $\g$ of complexity at most $C$ with respect to $\cX$ and let $U$ be a rational subspace of a distinct dimension $O(1)$ Lie algebra $\h$ of complexity at most $C$ with respect to a rational basis $\cY$. Then
\begin{enumerate}
\item $S + T$ has complexity at most $C^{O(1)}$,
\item $[S,T]$ has complexity at most $M^{O(1)}C^{O(1)}$, where $M$ is an upper bound on the height of the largest (rational) structure constant of $\g$ with respect to $\cX$,
\item $S \otimes U$ has complexity at most $C^{O(1)}$ in $\g \otimes \h$ with respect to $\cX \otimes \cY$.

\end{enumerate} 
\end{lemma}

\begin{lemma}[Complexity of maps under lifting and restriction]\label{l:comp-lift-restrict}
Let $A,B, C>1 $ be integers and let $S,T$ be rational subspaces of $\g$ of complexity at most $C$ with respect to $\cX$ such that $S \leq T$.  Then 
\begin{enumerate}
\item  $\cX'_{S,\cX}$ has index $C^{O(1)}$ in $\cX$, 
\item  if $l$ is an integer linear map on $T$ of complexity at most $A$, then $l|_S$ is an integer linear map of complexity at most $C^{O(1)}A$  on $S$,
\item if $l$ is an integer linear map on $S$ of complexity at most $B$ then there is an integer linear map $\tilde l$ on $T$ of complexity at most $C^{O(1)}B$ such that $\tilde l|_S = Kl$, where $K = C^{O(1)}$.

\end{enumerate}
\end{lemma}

Example \ref{ex:irrat} demonstrated that there are elements which are $(A,\eps)$-linearly irrational with respect to more than one subspace. In this example, the nontrivial subspace was of complexity $A^{O(1)}$. The following lemma gives a sense in which this is the only way in which such a phenomenon can occur. 
\begin{lemma}\label{l:unique-S}
Let $A>1$. Let $a\in \g$ and let $S,T$ be subspaces of $\g$ of complexity at most $A$ which contain $a$. There is a constant $k>1$ (which may depend on $\dim \g$) such that for every $\eps > 0$ if $a$ is $(A^k,\eps)$-linearly irrational in $T$ then $T \leq S$. 
\end{lemma} 
\begin{proof}
We show that $S^\perp \leq T^\perp$. Let $l: \g \to \R$ have complexity at most $A$ and lie in $S^\perp$. Since $a \in S$ we have $l(a)=0$. Furthermore $l$ has complexity at most $A^{O(1)}$ with respect to $\cX'_{\cX,T}$ by Lemma \ref{l:comp-lift-restrict} and so for $k=O(1)$ sufficiently large, this contradicts the $(A^k,\eps)$-linear irrationality of $a$ in $T$, unless $l \in T^\perp$. Letting $\g^\ast_{\cX}(A)$ denote the set of integer linear maps of complexity at most $A$ with respect to $\cX$, we have so far that $S^\perp \cap \g^\ast_{\cX}(A) \subset T^\perp$ and taking the vector space closure on both sides we have $\spa_\R\{S^\perp \cap \g^\ast_{\cX}(A)\} \subset T^\perp$. But by the definition of the complexity of $S$, we have that $S^\perp$ may be written as the span of complexity $\leq A$ maps and so $S^\perp = \spa_\R\{S^\perp \cap \g^\ast_{\cX}(A)\}$. This completes the proof.   
\end{proof}

In particular, there is a polynomial threshold above which linear irrationality may be obtained on at most one rational subspace of bounded complexity. In general, however, the sequence of subspaces on which we are asserting linear irrationality is not determined by the polynomial sequence and so our counting lemma will take this as additional input.

Similarly to the use of Leibman's criterion for equidistribution in the qualitative setting, we will introduce here Green and Tao's criterion for quantitative equidistribution. First a quantitative version of additive irrationality; see Subsection \ref{ss:smooth} for a definition of the $||\cdot||_{C^\infty[\vect N]}$-norm. In the following we use $\Hom_\cX(\g,\R)$ to denote the Lie algebra homomorphisms from $\g$ to $\R$ which map $\cX$ to $\Z$.

\begin{definition}[Quantitative additive irrationality]
Let $A > 1$, let $D$ be a positive integer, let $N_1,\ldots,N_D\geq 1$ be integers and let $[\vect N]$ denote the set $[N_1]\times \cdots \times [N_D]$. A polynomial sequence $p$ on $\g$ is $(A, \vect N)$-additively irrational with respect to $\cX$ if for all nontrivial $\eta \in \Hom_\cX(\g,\R)$ of complexity at most $A$, we have $||\eta \circ p ||_{C^\infty[\vect N]} > A$.
\end{definition}

\begin{remark}
Though we have not done so here, one may develop the various notions of irrationality in a transparently analogous way to highlight that these notions essentially only differ by the subalgebras of $\g$ on which the set of maps under consideration are required to vanish ($[\g,\g]$, $\h_i$, none for additive, filtration, linear irrationality respectively). One may define quantitative additive irrationality for elements of (rather than just polynomial sequences in) $\g$, and define quantitative filtration/linear irrationality for multivariate polynomial sequences; these definitions are not required in this document and so are not provided.
\end{remark}

Next we define quantitative equidistribution in $(\g,\cX)$, and defer to Appendix \ref{ss:lie-alg-equid} for more details.

\begin{definition}[Quantitative equidistribution in the Lie algebra]\label{d:quant-equid-alg}
Let $\cX$ be a rational basis in $\g$. Then a polynomial sequence $(p(\vect n))_{\vect n \in [\vect N]}$ is $\delta$-equidistributed in $(\g, \cX)$ if $(\exp (p(\vect n)))_{\vect n \in [\vect N]}$ is $\delta$-equidistributed in $G/\Gamma$, where $\Gamma$ is the smallest multiplicative lattice containing $\langle \cX \rangle = \spa_\Z \cX$.
\end{definition}

Green and Tao show \cite[Theorem p.3]{GT15} that quantitative additive irrationality implies quantitative equidistribution. Their result concerns equidistribution in a nilmanifold; we defer to Appendix \ref{a:group-to-alg} for a deduction of the following Lie algebra version.

\begin{thm}[Green and Tao's criterion for quantitative equidistribution in the Lie algebra]\label{t:gt-algebra}
Let $0<\delta<1/2$, let $N_1,\ldots,N_D\geq 1$ and let $[\vect N]$ denote the set $[N_1]\times \cdots \times [N_D]$. Let $\g$  be a  real nilpotent Lie algebra with dimension $d$, step $s$ and basis $\cX$ which is $1/\delta$-rational. Let $(p(\vect n))_{\vect n \in \Z^D}$ be a polynomial sequence in $\g$.  There are positive constants $c_{d,s,D}, c'_{d,s,D}$ such that at least one of the following is true:
\begin{enumerate}
\item there is some $N_i< \delta^{-c'_{d,s,D}}$, 
\item $(p(\vect n))_{{\vect n \in [\vect N]}}$ is $\delta$-equidistributed in $(\g, \cX)$, 
\item $p$ is not $(\delta^{-c_{d,s,D}}, \vect N)$-additively irrational in $(\g,\cX)$. 
\end{enumerate}
Furthermore, when $N_1 = \cdots = N_D$, the first option above may be removed. 
\end{thm}

We are ready to prove the main theorem of this section. We only prove a result for averages over a cube $[N]^D$ (i.e. when $N_1 = \cdots = N_D = N$), but it should be clear how to obtain a result for the more general case in which not all sides have equal length.

\begin{thm}[Quantitative counting lemma, Theorem \ref{t:main-quant}]\label{t:pr-main-quant}
Let $A,M\geq 2$. Let $\Psi = (\psi_1, \ldots, \psi_t)$ be a collection of linear forms each mapping $\Z^D$ to $\Z$. Let $\g$ be a real nilpotent Lie algebra with rational basis $\cX$, dimension $d$ and rational structure constants of height at most $M$. Let $S_\bullet$ be a sequence of rational subspaces of complexity  $M$ in $\g$. Suppose that $p(n)$ is an $(A,N)$-linearly irrational polynomial sequence  of degree $s$ in $S_\bullet$. Then there are constants $0<c_1,c_2 = O_{d,s,D,t,\Psi}(1)$ such that the polynomial sequence $p^\Psi(\vect x):= (p(\psi_1(\vect x)), \ldots, p(\psi_t(\vect x))$ is $O(M^{c_1}/A^{c_2})$-equidistributed on $\g^\Psi(S_\bullet)$ with respect to $\cX^\Psi$, the Hermite basis for $\g^\Psi(S_\bullet)$ in $(\g^t, \cX^t)$.
\end{thm}
\begin{proof} 
We let constants depend on $s,d,D,t,\Psi$. Suppose that $p$ is $(A,N)$-linearly irrational in $(S_i)_{i=1}^s$. Write $p(n) = \sum_{i=1}^s a_i n^i$, where $a_i$ is $(A,A/N^i)$-linearly irrational in $S_i$ (see Definition \ref{d:lin-irrat-poly}). 

Since $M$ bounds the complexity of $S_\bullet$ and the height of the structure constants of $\g$, the description of $\g^\Psi(S_\bullet)$ in Proposition \ref{p:gPsi-Sbullet} and repeated application of Lemma \ref{l:comp-operations} yields that $\g^\Psi(S_\bullet)$ is of complexity at most $M^{O(1)}$ in $\g\otimes \R^t$ with respect to $\cX  \otimes \cZ$, where $\cZ$ is the standard basis for $\Z^t$. Denote by $\cX^\Psi$ the Hermite basis for $\g^\Psi(S_\bullet)$ and $\cX \otimes \cZ$ and note that $\cX^\Psi$ has rational structure constants of height at most $M^{O(1)}$. Let $\delta>0$ be a parameter which we will choose later and suppose that $p^\Psi$ is not $\delta$-equidistributed in $\g^\Psi(S_\bullet)$ with respect to $\cX^\Psi$.

Invoking Theorem \ref{t:gt-algebra} with respect to the basis $\cX^\Psi$, we obtain a nontrivial $\eta \in \Hom_{\cX^\Psi}(\g^\Psi(S_\bullet), \R)$ which has complexity at most $M^{O(1)}\delta^{-O(1)}$ and such that $||\eta\circ p^\Psi||_{C^\infty[N]^D}< M^{O(1)}\delta^{-O(1)}$. We have then, recalling the definition of the $C_\cM^\infty$ norms from Definition \ref{d:monomial-norms} and invoking Lemma \ref{l:change-basis}, that $||s!\eta \circ p^\Psi||_{C^\infty_\cM[N]^D} < M^{O(1)}\delta^{-O(1)}$. Now we compute 
\begin{equation}
s!\eta(p^\Psi(\vect{x})) = \sum_{i=1}^s s!\eta(a_i\otimes \Psi(\vect x)^i)  = \sum_{m \in \cM}\sum_{i=1}^s s!\eta(a_i\otimes v_{m,i})m(\vect{x}),
\end{equation}
where, as in Subsection \ref{ss:flag}, $v_{m,i}\in \Z^t$ is the coefficient of the monomial $m(\vect x)$ in the polynomial $\Psi^i(\vect x)$. Thus, by the definition of the $C^\infty_\cM[N]^D$ norm Definition \ref{d:monomial-norms}, we have for all $m,i$ that 
\begin{equation}\label{e:eta-a-v-bd}
||s!\eta(a_i \otimes v_{m,i})||_{\R/\Z} < M^{O(1)}\delta^{-O(1)}/N^i.
\end{equation} 

By Lemma \ref{l:comp-lift-restrict}, there is $\tilde \eta \in (\g\otimes \R^t)^\ast$  of complexity $M^{O(1)}\delta^{-O(1)}$ with respect to ${\cX \otimes \cZ}$ such that $\tilde \eta|_{\g^\Psi(S_\bullet)} = M^{O(1)}s!\eta$. It follows then that $\tilde \eta ( \cdot\otimes v_{m,i})$ is an integer linear map on $\g$ of complexity $M^{O(1)}\delta^{-O(1)}$ with respect to $\cX$ and so invoking Lemma \ref{l:comp-lift-restrict} once again, we have that $\tilde \eta ( \cdot\otimes v_{m,i})|_{S_i}$ has complexity  $M^{O(1)}\delta^{-O(1)}$ with respect to $\cX_i'$, the Hermite basis for $S_i$.

Now note that Equation (\ref{e:eta-a-v-bd}) implies that $||\tilde \eta(a_i \otimes v_{m,i})||_{\R/\Z} < M^{O(1)}\delta^{-O(1)}/N^i$ for all $i,m$. Therefore, if $A \geq M^{O(1)}\delta^{-O(1)}$ (the complexity of the map $\tilde \eta ( \cdot\otimes v_{m,i})|_{S_i}$) and $A \geq M^{O(1)}\delta^{-O(1)}$ so that $||\tilde \eta(a_i \otimes v_{m,i})||_{\R/\Z}  < A/N^i$, then the $(A,A/N^{i})$-linear irrationality of $a_i$ in $S_i$ implies that  $\tilde \eta(\cdot  \otimes v_{m,i})$ is trivial on $S_i$, i.e. $\tilde \eta(S_i \otimes v_{m,i}) = 0$. Since the $\{v_{m,i}\}_{m \in \cM}$ span $V^i$, we have $\tilde \eta(S_i \otimes V^i) = 0$ and so from definitions $\eta(S_i\otimes V^i) = 0$. Since $\eta$ is a Lie algebra homomorphism to an abelian Lie algebra, $\eta$ vanishes on the smallest Lie subalgebra of $\g^\Psi(S_\bullet)$ containing the subspaces $S_i \otimes V^i$, which is $\g^\Psi(S_\bullet)$ itself. That is, $\eta$ is trivial on $\g^\Psi(S_\bullet)$, a contradiction. So $p^\Psi$ is indeed (upon solving for $\delta$) $(M^{O(1)}/A^{O(1)})$-equidistributed in $\g^\Psi(S_\bullet)$.
\end{proof}

\section{The factorisation theorem for linear irrationality}\label{s:fact}

As in the qualitative setting, a counting lemma is of little use without a method to reduce a general polynomial sequence to something it can handle. Our main result of this section is Theorem \ref{t:quant-factorise}, a factorisation theorem for linear irrationality in much the same spirit as \cite[Lemma 2.10]{GT10} and \cite[Theorem 1.19]{GT12}, which may be viewed as factorisation theorems for the notions of filtration and additive irrationality respectively. Throughout this section we assume $d,s,t,D$ are of bounded size and let constants depend on them. As we did in the qualitative setting, we begin with an additive decomposition.

\subsection{Additive decomposition for linear irrationality}
In this subsection we prove an additive decomposition of vector space elements into linearly irrational, rational and small parts (Proposition \ref{p:linear-decomp}), and as a corollary obtain a similar decomposition for polynomial sequences (Corollary \ref{c:poly-linear-decomp}).

\begin{definition}[Quantitative relative linear irrationality]\label{d:li-quotient}
Let $A>1$ and $\eps >0$, let $\cX$ be a rational basis for $\g$, and let $U, T$ be rational subspaces. An element $a \in U$ is $(A,\eps)$-\textit{linearly irrational in $U\bmod T$  with respect to $\cX$} if, for all integer linear maps $l$ which have complexity at most $A$ in $U$ with respect to $\cX'_{\cX,U}$, which are nontrivial on $U$ and which vanish on $U\cap T$, we have that $||l(a)||_{\R/\Z}\geq \eps$.
\end{definition}

\begin{definition}[Quantitative rationality]\label{d:quant-rat}
Let $a \in \g$, $A\in \N^+$. We say that $a$ is $A$-rational with respect to $\cX$ if $a\in \g_\Q$ and, when written with respect to the basis $\cX$, the vector for $a$ has entries in $\frac{1}{A}\Z$. A polynomial sequence $p$ in $\g$ is $A$-rational if all of its coefficients are $A$-rational.
\end{definition}

\begin{definition}[Quantitative smallness]\label{d:quant-small}
Let $\eps>0$. An element $a \in \g$ is $\eps$-small with respect to $\cX$ if $||a||_\infty \leq \eps$, where $||\cdot||_\infty$ is the $\ell_\infty$ norm of $a$ with respect to the basis $\cX$. A polynomial sequence  $(p(\vect n))_{\vect n \in [N]^D}$ in $\g$ is $(A,N)$\textit{-small} with respect to $\cX$ if, for all monomials $m \in \cM$, we have that $c_m$ is $A/N^{\deg m}$-small, where $c_m$ is the coefficient of $m(\vect n)$ in $p$.
\end{definition}

The following decomposition of vector space elements into linearly irrational, rational and error terms allows us to perturb an element which is linearly irrational in a quotient space by elements of the quotient to obtain an element which is linearly irrational in the original space. Note that this result includes the case where the quotient space is trivial (that is, $T=U$ in the upcoming notation), whereupon we obtain a decomposition of an arbitrary element in the vector space.

\begin{prop}[Quantitative decomposition of vector space elements]\label{p:linear-decomp}
Let $A>1$ be an integer, let $\cX$ be a rational basis for $\g$, let $U,T$ be rational subspaces of $\g$, each of complexity $A^{O(1)}$. Then there is a constant $c>1$ (which may depend on $\dim U, \dim T$) such that the following is true. Let $0 < \eps < 1/2$ and $a \in U$, and suppose that $a$ is $(A^c,A^c\eps)$-linearly irrational in $U\bmod T$. There are elements $a_s \in U \cap T$, $a_r \in (U\cap T)_\Q$ such that  $a_r$ is $A^{O(1)}$-rational, $a_s$ is  $ A^{O(1)}\eps$-small and $a- a_r- a_s$ is $(A,\eps)$-linearly irrational in a subspace  $U' \leq U$ of complexity $A^{O(1)}$ with respect to $\cX$. 
\end{prop}
\begin{proof}
If $a$ is $(A,\eps)$-linearly irrational in $U$ then we are done and so supposing it is not, we have an integer linear map $l$ of complexity at most $A$ with respect to the Hermite basis $\cX'_{U, \cX} = \cX'$ which is nontrivial on $U$ such that $||l(a)||_{\R/\Z} < \eps$. Note also that $l$ must be nontrivial as a map on $U\cap T$ due to the linear irrationality of $a$ in $U\bmod T$. 
This is the $m=1$ case of the following setup; our proof proceeds inductively on $m$. 

Let $L=L_m$ be an $m \times \dim U$ integer matrix whose rows are linearly independent, integer linear maps $U \to \R$ with respect to the basis $\cX':= \cX'_{U,\cX}$ and which possess the following properties. Firstly, the projection of the rows are linearly independent as maps in $(U\cap T)^\ast$.
Furthermore, for each row $l$, we have that $l$ has complexity $ A^{O(1)}$ with respect to $\cX'$ and that $||l \cdot a||_{\R/\Z} \leq A^{O(1)}\eps$.  
We will show below that either:
\begin{enumerate}
\item there are elements $a_s \in U\cap T$, $a_r \in (U \cap T)_\Q$ such that with respect to $\cX'$,  $a_r$ is $ A^{O(1)}$-rational, $||a_s||_\infty = A^{O(1)}\eps$ and $a- a_r- a_s$ is $(A,\eps)$-linearly irrational in $\ker L \leq U$ (i.e. the proof is complete upon observing that with respect to $\cX$,  $a_s$ is still $A^{O(1)}\eps$-small and $a_r$ is still $A^{O(1)}$-rational), or
\item we may find some $L_{m+1}$ which fits the description above with $m+1$ in place of $m$. 
\end{enumerate}
Note that in situation (1), we set $U' = \ker L$. One may use the fact that the rows of $L$ are linearly independent as maps in $(U \cap T)^\ast$ to show that $\dim (\ker L + T) = \dim(U + T)$. We omit the details which are just linear algebra. In situation (2), the constants implicit in the $O(1)$ notation may change when we pass from $m$ to $m+1$. However this can occur only finitely many times since $\dim (U\cap T) = O(1)$ and the rows of each $L_m$ are linearly independent in $(U\cap T)^\ast$. In particular, we must, at some point, found ourselves in situation (1), completing the proof. It remains to prove the above claim.

Let $L = L_m$ be as above. For $i=1, \ldots, m$, we have by supposition that $l_i\cdot a = n_i + e_i$, for some integer $n_i$ and $0 < |e_i| < \eps$. We may write this simultaneously for $i=1,\ldots ,k $ as $L a = n + e$ where $n,e$ are vectors consisting of the $n_i,e_i$ respectively.  

We claim that there is a solution to $Lx = n$ with $x \in U\cap T$ and such that each entry of $x$ with respect to $\cX'$ lies in $(\frac{1}{A^{O(1)}})\Z$; we will let this be $a_r$. 
Recall that by the induction statement $L$ has full rank as a family of linear maps on $U \cap T$. Let $M$ be the change of basis matrix such that $[L]_{\cX'} = [LM]_{\cX''}$, where $\cX'' = \cX''_{\cX, U\cap T}$ is the Hermite basis for $U\cap T$ in $\cX$. Note that $M$ is an integer matrix since $\langle \cX' \rangle = U\cap \langle \cX \rangle \supset \cX''$, and further, complexity bounds ensure that $M$ has entries of size at most $A^{O(1)}$. Let $L'$ be a square submatrix of $LM$ obtained by choosing $m$ leading columns. Then it is a standard linear algebra fact that the simultaneous equations $L'x' = n$ are solvable in $\frac{1}{|\det L'|}\Z$. Observing that $|\det L'| = A^{O(1)}$ and then padding the solution $x'$ with $\dim (U\cap  T) - m$ zero entries to form an appropriate vector $\tilde x$ yields a solution $LM\tilde x = n$ in $((\frac{1}{A^{O(1)}})\Z)^{\dim U\cap T}$. Setting $x = M\tilde x$ proves the claim made at the beginning of the paragraph. 

Next, we claim that there is a solution to $Lx = e$ with $x \in U\cap  T$ and $||x||_\infty = A^{O(1)}\eps$. This may be proven, for example, by choosing $L'$ as above,  setting $x' = L'^{-1}e$, upper bounding $||L'^{-1}||_\infty$ directly using using Cramer's rule and then augmenting $x'$ with zero entries to form a vector $\tilde x$ with $LM\tilde x = e$.  Let $a_s = M\tilde x \in U\cap T \subset U$ so that $La_s = e$ and $||a_s||_\infty = A^{O(1)}\eps$ (where here the $\ell_\infty$ norm is of course with respect to the basis $\cX'$).

Now, $a_p := a-a_r-a_s$ is contained in $\ker L$. If $a_p$ is $(A,\eps)$-linearly irrational in $\ker L$ then we are in situation (1) and we are done, so suppose not. Then there is some nontrivial integer linear map $l$ of complexity at most $A$ in $\ker L$ such that $||l (a_p)||_{\R/\Z} < \eps$. Also, it follows from definitions that  $\ker L$ is of complexity $ A^{O(1)}$  and so by Lemma \ref{l:comp-lift-restrict} we may find an integer linear map $\tilde l \in U^\ast$ of complexity $ A^{O(1)}$ such that $\tilde l|_{\ker L} = A^{O(1)}l$. Furthermore since $l$ is nontrivial on $\ker L$, we have that $\tilde l$ does not lie in the row span of $L$. We claim that in fact (the projection of) the rows of $L$ and $\tilde l$ are linearly independent in $(U\cap T)^\ast$. Suppose not for a contradiction with the $(A^c,A^c\eps)$-linear irrationality of $a$ in $U\bmod T$. Then there is some linear combination $b\tilde l + \sum_{i=1}^m b_il_i =0\in (U \cap T)^\ast$, where $b \ne 0$ because $\{l_1,\ldots,l_m\}$ are linearly independent in $(U\cap T)^\ast$. In fact, since the entries of each $l_i, \tilde l$ are integers of size $A^{O(1)}$, one may check e.g.\ by standard Gaussian elimination that we can insist that the constants $b_i,b$ are themselves integers of size $ A^{O(1)}$ (with a different constant $O(1)$ which may depend on the dimension of the space). For these $b,b_i$, let $l' = b\tilde l + \sum_{i=1}^m b_il_i  \in (U \cap T)^\perp \subset U^\ast$. Then $l'(a) = l'(a_p)$ (since $a_r,a_s\in U \cap T$), $l'(a_p) = b\tilde l(a_p)$ (since $a_p \in \ker L$), and $b\tilde l(a_p) = bA^{O(1)}l(a_p)$ (by the construction of $\tilde l$ and since $a_p \in \ker L$). Thus \[||l'(a)||_{\R/\Z} = ||bA^{O(1)}l(a_p)||_{\R/\Z} \leq bA^{O(1)}||l(a_p)||_{\R/\Z} < A^{O(1)}\eps.\]
However, $l'$ is of complexity $A^{O(1)}$ with respect to $\cX'$. 
For $c$ chosen suitably large, this contradicts the $(A^c, A^c\eps)$-linear irrationality of $a$ in $U\bmod T$. Thus we have that (the projection of) the rows of $L$ and $\tilde l$ are linearly independent in $(U \cap T)^\ast$.

Finally, let $C$ be the smallest positive integer such that $C\tilde l (a_{r}) \in \Z$; note that $C = A^{O(1)}$. Then we compute 
\begin{align*}
||C\tilde l(a)||_{\R/\Z} &\leq ||C\tilde l (a_{p})||_{\R/\Z} +  ||C\tilde l(a_{r})||_{\R/\Z} + ||C\tilde l(a_{s})||_{\R/\Z}\\
& \leq CA^{O(1)}\eps + 0 + CA^{O(1)}\eps \leq A^{O(1)}\eps.
\end{align*}
One may then form $L_{m+1}$ by adding the row $C\tilde l$ to the bottom of $L=L_m$. This puts us in situation (2) and completes the proof.
\end{proof}

\begin{cor}[Quantitative additive decomposition of polynomial sequences]\label{c:poly-linear-decomp}
Let $A, N>1$ be integers, let $U_\bullet = (U_i)_{i=1}^s$ be a sequence of rational subspaces of $\g$ each of complexity $A^{O(1)}$ and let  $T$ be a rational subspace of $\g$ which is also of complexity $A^{O(1)}$.  There is a constant $c>1$ (which may depend on $\dim U$, $\dim T$) such that the following is true. Let $p$ be a polynomial sequence defined by $p(n) = \sum_{i=1}^s a_in^i$ such that $a_i \in U_i$ for all $i$ and which is $(A^c,N)$-linearly irrational in $U_\bullet\bmod T$.\footnote{The operations are conducted pointwise on the sequence i.e. this may be translated to mean that each coefficient $a_i$ is $(A^c,A^c/N^i)$-linearly irrational in $U_i\bmod T$.} Then there are polynomial sequences $e,r$ in $U_\bullet \cap T$ of degree $s$ and a sequence of subspaces $(U'_i)_{i=1}^s \leq U_\bullet$ with each $U'_i$ of complexity $A^{O(1)}$ such that $e$ is $(A^{O(1)},N)$-small, $r$ is $A^{O(1)}$-rational, and $p-e-r$ is $(A,N)$-linearly irrational in $U'_\bullet:= (U'_i)_{i=1}^s$. Furthermore, $\dim(U'_i + T) = \dim(U_i + T)$ for all $i$.

\end{cor}
\begin{proof}
Let $p(n) = \sum_{i=1}^s a_i n^i$. If $p$ is $(A,N)$-linearly irrational in $U_\bullet$ then we are done so suppose it is not and conduct the following argument for each index $i$  such that $a_i$ fails to be $(A,A/N^i)$-linearly irrational in $U_i$. By definition we have that each such $a_i$ is $(A^c,A^c/N^i)$-linearly irrational in $U_i\bmod T$. Trivially then, $a_i$ is $(A^{c-1},A^c/N^i)$-linearly irrational in $U_i\bmod T$. Now invoke Proposition \ref{p:linear-decomp} for $a_i$ and let $c_i$ be a constant (obtained from the statement of Proposition \ref{p:linear-decomp}) so that we may find $a_{e,i}\in U_i \cap T$ which is $(A^{O(1)}/N^i)$-small and $a_{r,i}\in U \cap T$ which is $A^{O(1)}$-rational such that $a_i' := a_i - a_{r,i} - a_{e,i}$ is $(A^{(c-1)/c_i},A/N^i)$-linearly irrational in some subspace $U_i'\leq U_i$ of complexity $A^{O(1)}$. If $(c-1)/c_i\geq 1$ then $a_i'$ is $(A,A/N^i)$-linearly irrational in $U_i'$; insist that $c$ is chosen so that this is the case. Let $r(n) = \sum_{i=1}^s a_{r,i}n^i$ and $e(n) = \sum_{i=1}^s a_{e,i}n^i$. The result follows.
\end{proof}

\subsection{Multiplicative factorisation for strong irrationality}

The goal of this subsection is to prove Theorem \ref{t:quant-factorise}. This may be viewed as analogous to Green and Tao's factorisation theorems for polynomial sequences (see \cite{GT10}, \cite{GT12}) for the notion of linear irrationality. We will also want to control the following notion of irrationality which comes from \cite{GT10} (see in particular \cite[Appendix A]{GT10}).

In the following one will need to recall the definition of the subalgebras $\h_i$: Definition \ref{d:hi}.

\begin{definition}[Filtration irrationality]\label{d:fi}
Let $A>1, \eps >0$. We say $a\in \g_i$ is $(A,\eps)$-filtration irrational in $\g_i$ with respect to a rational basis $\cX$ for $\g$ if $||l(a)||_{\R/\Z} \geq \eps$ for all $l\in \g_\Q^\ast$ which are nontrivial on $\g_i$, have complexity at most $A$ on $\g_i$ and which vanish on $\h_i$.
\end{definition}

One may equivalently define $a\in \g_i$ to be $(A,\eps)$-filtration irrational in $\g_i$ if and only if it is $(A,\eps)$-linearly irrational in $\g_i\bmod \h_i$, as per Definition \ref{d:li-quotient}.

\begin{definition}[Filtration irrationality for polynomial sequences]\label{d:fi-poly}
Let $A,N >1$ and let $\g$ have rational basis $\cX$. Let $p$ be a polynomial sequence adapted to $\g_\bullet$ which is defined by $p(n) := \sum_{i=1}^s a_i n^i$. We say that $p$ is \textit{$(A,N)$-filtration irrational} if $a_i$ is $(A,A/N^i)$-filtration irrational in $\g_i$ for every $i$. Furthermore we say that $p$ is \textit{$(A,N)$-filtration irrational in $\g\bmod \g_j$} if $a_i$ is $(A,A/N^i)$-filtration irrational in $\g_i$ for $i=1,\ldots, j-1$. 
\end{definition}

Before proving Theorem \ref{t:quant-factorise}, we need some lemmas. Henceforth, for a univariate polynomial sequence $q$, let $(q)_j$ denote the coefficient of the term of degree $j$ in $q$. The following equation is an easy consequence of definitions. 
 
\begin{lemma}\label{l:mult-to-add-modh} 
For polynomials $q, q'$ adapted to $\g_\bullet$ we have 
\[ (q \ast q')_{j} = (q)_j + (q')_j \mod \h_{j},\]
for all $j$.
\end{lemma}

\begin{lemma}\label{l:irrationality-of-shift}
Let $A>R> 1$ be integers and let $\delta,\eps > 0$ be such that  $\delta - (\dim \g_i) A\eps > 0$. Suppose that $a, r, s \in \g_i$ such that $a$ is $(A,\delta)$-filtration irrational in $\g_i$, $r$ is $R$-rational in $\g_i$ and $s$ is $\eps$-small. Then $a-r-s$ is $(A/R, (\delta- (\dim \g_{i}) A\eps)/R)$-filtration irrational in $\g_i$.
\end{lemma}
\begin{proof}
Let $\beta = (\delta - (\dim \g_i) A\eps )/R$ and $b = a-r-s$. We prove the contrapositive: suppose $b$ is not $(A/R,\beta)$-filtration irrational. We may thus choose an integer linear map $l$, nontrivial on $\g_i$, which vanishes on $\h_i$ and with complexity $A/R$ in $\g_i$ such that $||l(b)||_{\R/\Z} < \beta$. Then $Rl$ has complexity $A$ on $\g_i$ and 
\[||Rl(a)||_{\R/\Z} \leq ||Rl(b)||_{\R/\Z} + ||Rl(r)||_{\R/\Z} + ||Rl(s)||_{\R/\Z} < R\beta + 0 + (\dim \g_i) A\eps = \delta,\]
so $a$ is not $(A,\delta)$-filtration irrational.
\end{proof}

Recall that we have already proven Proposition \ref{p:qual-factorise}, a qualitative version of the following theorem which may be of interest to a reader wishing to extricate the important parts of the upcoming proof.

\begin{thm}[Quantitative factorisation of polynomial sequences]\label{t:quant-factorise}
Let $A,M\geq 2$ and $d,s,t\geq 1$ be integers such that $AM> d$. Let $\g$ be a real nilpotent Lie algebra with filtration $\g_\bullet$, rational basis $\cX$, dimension $d$ and rational structure constants of height at most $M$. Let $p$ be a polynomial sequence adapted to $\g_\bullet$ and sequence of subspaces $S_\bullet\leq \g_\bullet$. Finally, insist that $S_\bullet, \g_\bullet$ are of complexity $A^{O(1)}$ in $\g$, where here and henceforth $O(1)$ terms may depend on $d,s$. Then we may write $p = e \ast p' \ast r$ where $e,p',r$ are polynomial sequences in $\g$ adapted to $\g_\bullet$, $p'$ is $(A,N)$-linearly irrational in a sequence of subspaces $S'_\bullet$ of complexity $A^{O(1)}$, $r$ is $(AM)^{O(1)}$-rational, and $e$ is $((AM)^{O(1)},N)$-small. Furthermore:
\begin{itemize}
\item writing $p(n) = \sum_{i=1}^s a_in^i$, if $j$ is the smallest integer such that $a_j$ is not $(A,N)$-linearly irrational in $S_j$, then we may take $S_k'=S_k$ for all $k=1,\ldots,j-1$ and have $S_j'$ strictly contained in $S_j$,
\item there is a constant $c>1$ such that if $B >AM$ and $p$ is $(B^c,N)$-filtration irrational then $p'$ is $(B,N)$-filtration irrational.
\end{itemize} 

\end{thm}
\begin{proof}
The factorisation process proceeds inductively. We will demonstrate the first couple of steps, from which the general process should be clear. 

If $p$ is $(A,N)$-linearly irrational in $S_\bullet$ then we are done trivially. Suppose not, let $b\geq 1$ be a constant to be chosen later, and write $p(n) = \sum_{i=1}^s a_in^i$ and let $j$ be the smallest integer such that $a_j$ is not $(A^b,N)$-linearly irrational in $S_j$ (since $b \geq 1$, such a $j$ certainly exists). We begin by applying Corollary \ref{c:poly-linear-decomp} to the degree $\geq j$ terms of $p$ with $U_\bullet = S_\bullet$ and $T_\bullet = (\g,\ldots, \g)$ (note that $p$ is arbitrarily linearly irrational in $U_\bullet \bmod T$) to obtain $p_1 := p - \tilde e_1 - \tilde r_1$, where $\tilde e_1, \tilde r_1$ are polynomial sequences adapted to $\g_\bullet$ with  all terms of degree at least $j$, $\tilde e_1$ is $((AM)^{O(1)},N)$-small, $\tilde r_1$ is $A^{O(1)}$-rational and $p_1$ is $(A^{b},N)$-linearly irrational  in some sequence of subspaces $S^{(1)}_\bullet$ with $S^{(1)}_k = S_k$ for $k<j$ and $S^{(1)}_j < S_j$. Then set \[\tilde p_2 := \tilde e_1^{-1}\ast p \ast \tilde r_{1}^{-1} = p - \tilde e_1 - \tilde r_1 - \frac{1}{2}([\tilde e_1,p] + [\tilde e_1,r_1] + [p,\tilde r_1]) - \ldots.\] 
Note that $\tilde p_2 = p_1 \mod \g_2$ and furthermore that $(\tilde p_2)_k = (p_1)_k$ for all $k\leq j$ (since every term coming from a Lie bracket will have degree at least $j+1$). The former fact implies that $\tilde p_2$ is $(A^b,N)$-linearly irrational in $(S^{(1)}_\bullet + \g_2)\bmod \g_2$.

Next we apply Corollary \ref{c:poly-linear-decomp} to the degree $> j$ terms of $\tilde p_{2}$ with $U_{\bullet} = S_{\bullet}^{(1)} + \g_{2}$ and $T = (\g_2,\ldots, \g_2)$.  We obtain $p_2:= \tilde p_2 - \tilde e_{2} - \tilde r_{2}$, where $\tilde e_2,\tilde r_2$ are polynomial sequences in $\g_{2}$, adapted to $\g_\bullet$ and with all terms of degree $>j$, $\tilde e_{2}$ is $(A^{O(1)},N)$-small, $\tilde  r_{2}$ is $A^{O(1)}$-rational, and $p_{2}$ is $(A^{b/c_2},N)$-linearly irrational (where $c_2$ is the constant from the statement of Corollary \ref{c:poly-linear-decomp}) in a sequence of subspaces $S_\bullet^{(2)}$  with  $S^{(2)}_k = S^{(1)}_k = S_k$ for $k<j$ and $S^{(2)}_j= S^{(1)}_j < S_j$.

If $s=2$ (i.e. $\g_3=0$), we conclude:   
\begin{align*}
p &= \tilde e_1 \ast \tilde p_{2} \ast \tilde r_1 \\ &= \tilde e_1 \ast (\tilde e_{2} + p_2 + \tilde r_2) \ast \tilde  r_1 \\ &= \tilde e_1 \ast \tilde e_{2} \ast  p_{2} \ast \tilde r_{2} \ast \tilde r_1,
\end{align*}
where in the final line we have used the fact that $\tilde e_{2}, \tilde r_{2} \in \g_{2}$, Baker-Campbell-Hausdorff and that $[\g,\g_{2}]\subset \g_{3}$. Furthermore Lemma \ref{l:poly-products} yields that $e_2 := \tilde e_1 \ast \tilde e_{2}$ is $((AM)^{O(1)},N)$-small and $r_2 := \tilde r_{2}\ast \tilde r_1$ is $(AM)^{O(1)}$-rational. 

Otherwise, set 
\[ \tilde p_3 := \tilde e_2^{-1}\ast \tilde p_2 \ast \tilde r_2^{-1} = \tilde p_2 - \tilde e_2 - \tilde r_2 - \frac{1}{2}([\tilde e_2,\tilde p_2] + [\tilde e_2, \tilde r_2] + [\tilde p_2, \tilde r_2]) - \ldots.\]
Note that $\tilde p_3 = p_2 \mod \g_3$ and furthermore that $(\tilde p_3)_k = (\tilde p_2)_k = (p_1)_k$ for all $k\leq j$. Apply Corollary \ref{c:poly-linear-decomp} to the degree $>j$ terms of $\tilde p_3$ to obtain a linearly irrational $p_3$ and proceed as above. 

If we begin with $b \geq \prod_{i=2}^{s} c_i$, then the process terminates with the polynomial sequence $p_{s}$ which is $(A,N)$-linearly irrational in $S_\bullet^{(s)}$, a sequence of complexity $A^{O(1)}$ subspaces of $\g$ with $S^{(s)}_k = S_k$ for all $k < j$ and $S^{(s)}_j = S^{(1)}_j < S_j$.

Now we proceed to the statement about filtration irrationality. Let $c$ be a large constant whose value we will decide upon later and let $b_1 := c$. Suppose that $p_i$ is $(B^{b_i},N)$-filtration irrational in $\g\bmod \g_{i+1}$ and recall our (implicit) definitions  
\begin{equation}\label{e:pi+1}
p_{i+1} := \tilde p_{i+1} - \tilde e_{i+1} - \tilde r_{i+1} := \tilde e_i^{-1} \ast (\tilde e_i +  p_i +  \tilde r_i) \ast  \tilde r_i^{-1} - \tilde  e_{i+1} - \tilde r_{i+1},\end{equation}
where $\tilde e_{i}, \tilde r_{i} \in \g_{i}$ and $\tilde e_{i+1}, \tilde r_{i+1} \in \g_{i+1}$, and 
\begin{equation}\label{e:eiri}
e_i := \prod_{j=1}^i \tilde e_j, \qquad r_i := \prod_{j=1}^i \tilde r_j.
\end{equation} 
It follows from (\ref{e:pi+1}) that $p_{i+1} = p_i$ modulo $\g_{i+1}$. For $j=1, \ldots, i$ we have that $\h_j \geq \g_{i+1}$ and so the $j$th coefficient of $p_{i+1}$ inherits $(B^{b_i},B^{b_i}/N^j)$-filtration irrationality from that of $p_i$. To check the filtration irrationality of $p_{i+1}$ in $\g\bmod \g_{i+2}$, it remains to check the filtration irrationality of the $(i+1)$th coefficient of $p_{i+1}$. To this end we analyse
\begin{align*} 
(p_{i+1})_{i+1} &= (e_i^{-1} \ast p \ast r_i^{-1} - \tilde e_{i+1} - \tilde r_{i+1})_{i+1}\\ &= (e_i^{-1} \ast p \ast r_i^{-1})_{i+1} - (\tilde e_{i+1})_{i+1} - (\tilde r_{i+1})_{i+1} \\
&= (p)_{i+1}  - (e_i)_{i+1} -  (r_i)_{i+1} - (\tilde e_{i+1})_{i+1} - (\tilde r_{i+1})_{i+1}  \mod \h_{i+1}, 
\end{align*}
where we have used Lemma \ref{l:mult-to-add-modh} in the second  line. It follows from definitions and Lemma \ref{l:poly-products} that both $(e_i)_{i+1}$ and $(\tilde e_{i+1})_{i+1}$ are $\frac{(AM)^{O(1)}}{N^{i+1}}$-small so that $(e_i)_{i+1} + (\tilde e_{i+1})_{i+1}$ is $\frac{(AM)^{O(1)}}{N^{i+1}}$-small by the triangle inequality. Also, both $(r_i)_{i+1}$ and $(\tilde r_{i+1})_{i+1}$ are $(AM)^{O(1)}$-rational so $(r_i)_{i+1} + (\tilde r_{i+1})_{i+1}$ is $(AM) ^{O(1)}$-rational. We have by supposition that $(p)_{i+1}$ is $(B^c,B^c/N^{i+1})$-filtration irrational, and so is trivially $(B^{c-k}, B^c/N^{i+1})$-filtration irrational for any constant $k>0$.  Invoking Lemma \ref{l:irrationality-of-shift} (and to do so recalling that $B > AM > \dim \g$), we have that $(p_{i+1})_{i+1}$ is $(\frac{B^{c-k}}{(AM)^{O(1)}}, \frac{B^{c}/N^{i+1} - (\dim \g)B^{c-k}(AM)^{O(1)}/N^{i+1})}{(AM)^{O(1)}})$-filtration irrational in $\g_{i+1}$, which in turn is  also $(B^{c-k - O(1)}, \frac{B^{c-O(1)} - B^{c-k+O(1)}}{N^{i+1}})$-filtration irrational  since $B>AM > \dim \g$. Setting $k=O(1)$ large enough depending on the other $O(1)$ terms, we may arrange for this coefficient to be $(B^{c - O(1)}, \frac{B^{c-O(1)}}{N^{i+1}})$-irrational. Therefore, $p_{i+1}$ is $(B^{b_{i+1}},N)$-filtration irrational where $b_{i+1}:= \min(b_i,c-O(1))$ and $c=O(1)$ has been chosen sufficiently large. By induction, if $c=O(1)$ is chosen sufficiently large, $p_{s}$ will indeed be $(B,N)$-filtration irrational.  This completes the proof.
\end{proof}

The previous result says that if we start with a sequence that is sufficiently filtration irrational, then we can factorise to a sequence which is both filtration irrational and linearly irrational.

\begin{definition}[Strong irrationality]\label{d:strong-irrat}
If a polynomial sequence $p$ is both $(A,N)$-filtration irrational in $\g_\bullet$ and $(A,N)$-linearly irrational in some sequence of rational subspaces $S_\bullet$ then we will say that it is $(A,N)$\textit{-strongly irrational} in $(\g_\bullet,S_\bullet)$.
\end{definition}

\section{The strongly irrational arithmetic regularity lemma}\label{s:arl}

In this section we prove our strongly irrational arithmetic regularity lemma Theorem \ref{t:strong-arl}. The main ingredients are Green and Tao's non-irrational regularity lemma and the factorisation theory we developed in the previous section. 

\begin{thm}[{\cite[Proposition 2.7]{GT10}}, Non-irrational arithmetic regularity lemma]\label{t:non-irrat-arl}
Let $N>1$, let  $f:[N]\to [0,1]$, let $s \geq 1$, let $\eps > 0$, and let $\cF: \R^+ \to \R^+$ be a growth function. Then there exists a quantity $M=O_{s,\eps,\cF}(1)$ and a decomposition
\[f = f_\nil + f_\sml + f_\unf \]
of $f$ into functions $f_\nil, f_\unf: [N] \to [-1,1]$ such that
\begin{enumerate}
\item ($f_\nil$ structured) $f_\nil$ is a degree $\leq s$ polynomial nilsequence of complexity $\leq M$,
\item ($f_\sml$ small) $||f_\sml||_{L^2[N]} \leq \eps$,
\item ($f_\unf$ very uniform)  $||f_\unf||_{U^{s+1}[N]} \leq 1/\cF(M)$,
\item (Nonnegativity)  $f_\nil$ and $f_\nil + f_\sml$ take values in $[0,1]$.
\end{enumerate}
\end{thm}

 In fact, Theorem \ref{t:strong-arl} produces a (virtual) nilsequence which is both linearly irrational and filtration irrational. Although the latter is not strictly necessary to determine distribution and indeed is omitted in Section \ref{s:examples}, it is useful in applications (where otherwise one might have to replicate a similar factorisation on a case-by-case basis). Throughout this section we assume that $d,s,D,t$ are of size $O(1)$.

The first step is to bootstrap our factorisation theory from the previous section to obtain a factorisation where $p'$ is (much) more strongly irrational (by an arbitrary growth function, say) than $e$ is small and $r$ is rational. To do so, we first factorise for (very) amplified irrationality using \cite[Lemma 2.10]{GT10} and then repeatedly factorise for linear irrationality using Proposition \ref{t:quant-factorise} and a reduction argument on the lexicographical ordering of the sequence $\dim(S_\bullet):=(\dim S_1, \ldots, \dim S_s)$. 

The following is what is needed from \cite{GT10}.

\begin{prop}[Factorising for amplified filtration irrationality, {\cite[Lemma 2.10]{GT10}}]\label{p:irrational-factorise}

Let $M_0>1$. Let $\g$  be a real nilpotent Lie algebra of dimension at most $M_0$  with filtration $\g_\bullet$ of length at most $M_0$. Let $\cX$ be a rational basis for $\g$ with structure constants of height at most $M_0$. Let $p$ be a polynomial sequence adapted to $\g_\bullet$  with $p(0)=0$. For any growth function $\cF: \R^+ \to \R^+$, there exists $M_0 \leq M \leq O_{M_0,\cF}(1)$ and a factorisation $p = e \ast p' \ast r$ where $e$ is $(O_M(1),N)$-small, where $r$ is $O_M(1)$-rational and where $p'$ is $(\cF(M),N)$-filtration irrational in a subfiltration $\g'_\bullet \leq \g_\bullet$ of complexity $O_M(1)$.\footnote{That is, for all $i$, $\g_i'$ is a subspace of $\g_i$ of complexity $O_M(1)$.} This filtration irrationality is measured with respect to $\cX'$, a Mal'cev basis\footnote{See \cite[Definition 2.1]{GT12}.} for $\g_\bullet'$ which may be written as a linear combination of $\cX$ with rational coefficients of height $O_M(1)$.
\end{prop}
\begin{proof}
The result \cite[Lemma 2.10]{GT10} is phrased in slightly different language and there is some checking to be done to show that their statement implies ours. We do this in Appendix \ref{a:fact}. 
\end{proof}

The following proposition is the ``strongly irrational'' analogue of the previous proposition.

\begin{prop}[Factorising for amplified strong irrationality]\label{p:amp-strong-irrat}
Let $M_0>1$. Let $\g$  be a real nilpotent Lie algebra of dimension at most $M_0$ with filtration $\g_\bullet$ of length at most $M_0$. Let $\cX$ be a rational basis for $\g$ with structure constants of height at most $M_0$. Let $p$ be a polynomial sequence adapted to $\g_\bullet$ with $p(0)=0$. For any growth function $\cF: \R^+ \to \R^+$, there exists $M_0 \leq M \leq O_{M_0,\cF}(1)$ and a factorisation $p = e \ast \tilde p \ast r$, where $e$ is $(O_M(1),N)$-small, where $r$ is $O_M(1)$-rational, and where $\tilde p$ is $(\cF(M),N)$-strongly irrational in some $(\g'_\bullet, S_\bullet)$ where $\g'_\bullet$ is a subfiltration of $\g_\bullet$ of complexity $O_M(1)$ in $\g$ and $S_\bullet$ is a sequence of subspaces in $\g'_\bullet$ of complexity $O_M(1)$. The filtration irrationality is measured with respect to $\cX'$, a Mal'cev basis for $\g_\bullet'$ where each element may be written as a linear combination of $\cX$ with rational coefficients of height at most $O_{M}(1)$.
\end{prop}
\begin{proof}
Let $\cF'$ be a growth function depending on $\cF$ and $M_0$ to be determined later. For the meantime, insist that $\cF' \geq (\cF\circ \cF)^c$ at all inputs, where $c$ is the constant from the statement of Theorem \ref{t:quant-factorise}. Invoke  Proposition \ref{p:irrational-factorise} to factorise $p = e \ast p' \ast r$ whereupon there is $M_0' = O_{M_0, \cF'}(1) = O_{M_0,\cF}(1)$ such that $p'$ is $(\cF'(M_0'),N)$-filtration irrational in some $(\g'_\bullet,\cX')$ of complexity $O_{M'_0}(1)$, $e$ is $O_{M_0'}(1)$-small and $r$ is $O_{M_0'}(1)$-rational. Note also that $\g'$ has rational structure constants of height $O_{M'_0}(1)$ with respect to $\cX'$ since the elements of $\cX'$ are rational linear combinations of the elements of $\cX$ with coefficients of height $O_{M_0'}(1)$. 

If $p'$ is $(\cF(M_0'),N)$-linearly irrational in $S_\bullet = \g'_\bullet$ then we are done upon setting $M := M_0'$. Otherwise, set $M_1:= \cF(M_0')=O_{M_0,\cF}(1)$ and invoke Theorem \ref{t:quant-factorise} to obtain a factorisation $p' = e_1 \ast p_1 \ast r_1$ where $e_1$ is $(O_{M_1}(1),N)$-small, $p_1$ is $(M_1,N)$-linearly irrational in some sequence of subspaces $S^{(1)}_\bullet \leq \g'_\bullet$ of complexity $O_{M_1}(1)$, and $r_1$ is $O_{M_1}(1)$-rational. Note that $p_1$ is $(M_1,N)$-linearly irrational in $S^{(1)}_\bullet$ but is not $(M_1,N)$-linearly irrational in $\g'_\bullet$, and so we may conclude from the first bullet point of Theorem \ref{t:quant-factorise} that $\dim (S_\bullet)$ (that is, the vector of dimensions) is lexicographically strictly less than $\dim(\g'_\bullet)$. Recall also that $\cF'(M_0')\geq \cF(\cF(M_0'))^c=\cF(M_1)^c$ and so the $(\cF'(M_0'),N)$-filtration irrationality of $p'$ allows us to conclude as per the second bullet point of Theorem \ref{t:quant-factorise} that $p_1$ is $(\cF(M_1),N)$-filtration irrational in $\g'_\bullet$. 

If in fact $p_1$ is $(\cF(M_1),N)$-linearly irrational in $S^{(1)}_\bullet$ then observe that $e_1$ and $r_1$ are $(O_{M_0,\cF}(1),N)$-small and $O_{M_0,\cF}(1)$-rational respectively with respect to $\cX$ and so we are done with $M=M_1$. Otherwise, set $M_2:=\cF(M_1)$, insist that $\cF' \geq (\cF \circ \cF \circ \cF)^c$ at all inputs, and proceed as before to obtain some $p_2$ which is $(M_2,N)$-linearly irrational on a sequence of subspaces $S^{(2)}_\bullet$ with $\dim(S^{(2)}_\bullet)$ lexicographically strictly less than $\dim(S^{(1)}_\bullet)$. Now we have $p = e_1 \ast e_2 \ast p_2 \ast r_2 \ast r_1$, where one may use Lemma \ref{l:poly-products} to see that $e_1 \ast e_2$ and $r_2 \ast r_1$ are are $(O_{M_0,\cF}(1),N)$-small and $O_{M_0,\cF}(1)$-rational respectively with respect to $\cX$. If in fact $p_2$ is $(\cF(M_2),N)$-linearly irrational then we are done; if not, proceed as above. Since $\g'$ has dimension at most $M_0$ and filtration of length at most $M_0$, the vector $\dim(S^{(i)}_\bullet)$ can lexicographically decrease at most $O_{M_0}(1)$ times. Thus the process must terminate at some $p_i$ which is $(\cF(M_i),N)$-strongly irrational in some $(\g',S^{(i)}_\bullet)$ of complexity $O_{M_i}(1)$, with $e_i$ which is $(O_{M_i}(1),N)$-small, and $r_i$ which is $O_{M_i}(1)$-rational. This proves the claim. 
\end{proof}

For convenience we will deal in this section only with (virtual) nilsequences for which the rational basis `passes through' the filtration $\g_\bullet$. Since the non-irrational arithmetic regularity lemma Theorem \ref{t:non-irrat-arl} produces nilsequences for which this restriction holds, we do not practically lose anything by making this simplifying assumption.

\begin{definition}
We will say that a basis $\cX:= \{X_1,\ldots, X_n\}$ for a vector space $S$ \textit{passes through} a sequences of rational subspaces $S_1 \geq S_2 \geq \ldots$ if $\{X_{n- \dim S_i+1}, \ldots, X_n\}$ is a basis for $S_i$ for all $i$. We may also say that each $S_i$ passes through $\cX$.
\end{definition}

\begin{definition}[Virtual nilsequence]\label{d:virtual-nilseq}
A \textit{virtual nilsequence} of complexity $M$ at scale $N$ is a function $\phi:[N]\to \C$ of the form $\phi(n)=f(p(n)\ast \Lambda, n \pmod q, n/N)$ where $q\leq M$ is a positive integer, $p$ is a polynomial sequence with respect a filtered nilpotent Lie algebra $\g_\bullet$ of dimension $d\leq M$ and step $s\leq M$,  $\Lambda$ is an additive lattice in $\g_\bullet$ with $\Z$-basis $\cX$  which passes through $\g_\bullet$ and for which $\g$ has rational structure constants of height at most $M$, and $f$ is Lipschitz with Lipschitz constant at most $M$.\footnote{As in \cite{GT10}, the underlying choice of metric does not matter too much. It may be taken to be the direct sum of the metrics on each of the factors where the metric on $\g$ is induced by the $\ell_\infty$ norm with respect to $\cX$ and the metric on $\Z/q\Z$ is obtained by embedding this group into $\R/\Z$.}

Let $S_\bullet$ be a sequence of subspaces of $\g_\bullet$ of complexity at most $M$. We will furthermore say that $\phi$ is $(A,N)$-filtration (resp. linearly/strongly) irrational in $\g_\bullet$ (resp. $(\g_\bullet, S_\bullet)$) if $p$ is $(A,N)$-filtration (resp. linearly/strongly) irrational in $\g_\bullet$ (resp. $(\g_\bullet, S_\bullet)$) with respect to $\cX$.
\end{definition}

\begin{remark}
In the above, $\Lambda$ may not be a multiplicative subgroup (i.e. with respect to the $\ast$ operation), in which case this definition is equivalent to asking that $f$ is automorphic with respect to the multiplicative closure of $\Lambda$ in $\g$. We see in Appendix \ref{ss:lie-alg-equid} that this multiplicative closure is indeed a multiplicative lattice which is not too much larger than $\Lambda$.
\end{remark}

Our conclusion to the argument follows a similar path to that of \cite[Theorem 1.2]{GT10} and may essentially be viewed as a Lie algebra analogue of what is done there. We note/caution that in the following proposition, the notion of complexity of a nilsequence used is that of Green and Tao, whereas the  notion of complexity of a virtual nilsequence is that of our Definition \ref{d:virtual-nilseq}. 

\begin{prop}\label{p:nil-is-virtual}
Let $s >0$, $M_0>1$, let $\cF:\R^+ \to \R^+$ be a growth function, and let $\phi: \Z \to [0,1]$ be a degree $\leq s$ nilsequence of complexity $\leq M_0$. Then there exists $M = O_{s,M_0,\cF}(1)$ such that $f$ (when restricted to $[N]$) is also a $(\cF(M),N)$-strongly irrational degree $\leq s$ virtual nilsequence of complexity $\leq M$. 
\end{prop}
\begin{proof}
As usual we let all constants depend on $s$. Write $\phi(n) = f(p(n))=f(p(n) \ast \log \Gamma)$ where $p$ is a polynomial sequence on some filtered nilpotent Lie algebra $\g_\bullet$ and where $f$ is automorphic with respect to some lattice $\Gamma \leq G$. Let $\cX_\mal$ be the Mal'cev basis which exhibits $\phi$ as having complexity $\leq M_0$. By \cite[Lemma A.8]{GT12}, there is an integer $q_0= M_0^{O(1)}$ such that $\cX := q_0\cX_\mal$ yields $\langle \cX \rangle \subset \log \Gamma$. Since $\cX_\mal$ has structure constants of height at most $M_0$, we have that $\cX$ has structure constants of height $M_0^{O(1)}$. Write $p = p(0)\ast (p(0)^{-1} \ast p)$, invoke Proposition \ref{p:amp-strong-irrat} on the bracketed polynomial (with $M_0$ replaced by $M_0^{O(1)}$ and for an unspecified growth function $\cF'$) and inherit the notation from the statement of that proposition so $\phi(n) = f(p(0) \ast e(n) \ast p'(n) \ast r(n)\ast \log \Gamma)$. From Lemma \ref{l:rat-period} we have that $r(n)\ast \log \Gamma$ is $q$-periodic, where $q=O_M(1)$ and we recall that $M=O_{M_0,\cF'}(1)$ from Proposition \ref{p:amp-strong-irrat}. Now define $\tilde f: \g \times \Z/q\Z \times \R \to [0,1]$ by 
\[\tilde f (x,s,y):= f(p(0)\ast e(Ny)\ast x \ast r(s)\ast \log \Gamma),\]
whenever $y \in \frac{1}{N}\Z$ and by Lipschitz extension to $\R$. Note that the $q$-periodicity of the orbit $r(n) \ast \log \Gamma$ yields that $\tilde f$ is indeed a well-defined function on its second argument. 

Next we claim that there exists a positive integer $l= O_M(1)$ such that $\tilde f(x \ast (l\langle \cX \rangle), s,y) = \tilde f(x, s,y)$, so that $\tilde f$ is automorphic with respect to the additive lattice $l \langle \cX \rangle$ in its first argument. Since $\langle \cX \rangle \subset \log \Gamma$, it suffices to find $l$ such that for all $n \in \Z$ we have $(l\langle \cX \rangle) \ast r(n) \subset r(n) \ast \langle \cX \rangle$. But this follows easily from the fact that $r(n)$ is $O_M(1)$-rational with respect to $\cX$, the Baker-Campbell-Hausdorff formula and our bound on the heights of the rational structure constants of $\g$ with respect to the basis $\cX$. Let $\Lambda = l\langle \cX \rangle = \langle l \cX \rangle$. Since $p'$ is $(\cF'(M),N)$-strongly irrational in $(\g_\bullet,S_\bullet)$ with respect to $\cX$, it is $(\gg_M \cF'(M),N)$-strongly irrational in $(\g_\bullet,S_\bullet)$ with respect to $l\cX$.

Using the Lipschitz property of $f$, it is also not difficult to show that $\tilde f$ is $O_M(1)$-Lipschitz, but we omit the details. The result then follows after replacing $M$ with a suitable quantity $O_M(1)$ and letting $\cF$ be sufficiently rapidly growing depending on $\cF'$.

\end{proof}

We are finally ready to recover the main result of this section.

\begin{thm}[Strongly irrational arithmetic regularity lemma, Theorem \ref{t:strong-arl}]\label{t:pr-strong-arl}
Let $f:[N]\to [0,1]$, let $s \geq 1$, let $\eps > 0$, and let $\cF: \R^+ \to \R^+$ be a growth function. Then there exists a quantity $M=O_{s,\eps,\cF}(1)$ and a decomposition
\[f = f_\nil + f_\sml + f_\unf \]
of $f$ into functions $f_\nil, f_\unf: [N] \to [-1,1]$ such that
\begin{enumerate}
\item($f_\nil$ structured) $f_\nil$ is a degree $\leq s$, $(\cF(M),N)$-strongly irrational virtual nilsequence of complexity $\leq M$,
\item($f_\sml$ small) $||f_\sml||_{L^2[N]} \leq \eps$,
\item($f_\unf$ very uniform) $||f_\unf||_{U^{s+1}[N]} \leq 1/\cF(M)$,
\item(Nonnegativity) $f_\nil$ and $f_\nil + f_\sml$ take values in $[0,1]$.
\end{enumerate}
\end{thm}
\begin{proof}
This is an easy corollary of Theorem \ref{t:non-irrat-arl} and Proposition \ref{p:nil-is-virtual}.
\end{proof}

\begin{remark}
Technically a strongly irrational virtual nilsequence consists of data $(p,\g,\cX,\g_\bullet,S_\bullet)$. However, recalling Lemma \ref{l:unique-S}, if $\cF$ is sufficiently rapidly growing, we may drop the data $S_\bullet$ since it may be uniquely recovered from the polynomial sequence $p$. Furthermore, by Lemma \ref{l:Ti+hi=gi} and Lemma \ref{l:maximal-Wi}, one may also recover the filtration $\g_\bullet$ from $p$. Thus, for sufficiently strongly irrational virtual nilsequences, the data $(p, \g, \cX)$ suffices.
\end{remark}

\section{Applications}\label{s:app}

\subsection{Recovering the flag counting lemma}\label{ss:recover}
This subsection is not an application as much as it is a sanity check (Corollary \ref{c:sanity}) that in the flag case, Green and Tao's arithmetic regularity lemma and counting lemma may be recovered by the results in this document. Its purpose is also to prove Lemma \ref{l:Ti+hi=gi} which is useful in other  applications.

Recall Definition \ref{d:hi}, that of the subalgebras $\h_i$ associated to the filtered Lie algebra $\g_\bullet$. The following is a quantitative version of Lemma \ref{l:si+hi-qual}.

\begin{lemma}\label{l:Ti+hi=gi}
Suppose that $p$ is $(A,N)$-filtration irrational in $\g_\bullet$ where the rational basis $\cX$ passes through $\g_\bullet$. Then, for each $i=1,\ldots, s$,  for any subspace $T_i\leq \g_i$ which contains $a_i$ and has complexity at most $A$, we have $T_i+\h_i = \g_i$.
\end{lemma}
\begin{proof}
From Definition \ref{d:hi} and the filtration on $\g_\bullet$ we have $\h_i \leq \g_i$ and so $T_i + \h_i \leq \g_i$. The slightly less trivial containment to show is $T_i + \h_i \supseteq \g_i$. Let $Q=Q(A)$ be the $\R$-span of all complexity $\leq A$ linear maps on $\g_i$.\footnote{Since $\cX$ passes through $\g_\bullet$ the complexity of a map on $\g$ is equal to that of its restriction on any $\g_i$.} Since $a_i$ is $(A,A/N^i)$-filtration irrational, we in particular have that $l(a_i)\ne 0$ for all nontrivial $l \in Q$ which vanish on $\h_i$. Thus in $\g_i^\ast$ we have $a_i^\perp \cap Q \cap \h_i^\perp = \{0\}$ so $(a_i^\perp \cap Q)^\perp + \h_i = \g_i$. But $\g_{i\cX}^\ast(A) \subset Q$ where $\g_{i\cX}^\ast(A)$  is the set of integer linear maps on $\g_i$ of complexity at most $A$, and so $(a_i^\perp \cap Q)^\perp \subset (a_i^\perp \cap \g_{i\cX}^\ast(A))^\perp$. However, $(a_i^\perp \cap \g_{i\cX}^\ast(A))^\perp$ is, by definition, the smallest rational subspace of $\g_i$ which contains $a_i$ and has complexity at most $A$. Therefore, $T_i \supset (a_i^\perp \cap \g_{i\cX}^\ast(A))^\perp$ and ultimately, $T_i + \h_i \supset \g_i$.
\end{proof}

\begin{cor}[Theorem \ref{t:flag-cl}]\label{c:sanity}
Let $A,M \geq 2$. Let $\Psi = (\psi_1, \ldots, \psi_t)$ be a flag system of linear forms, each mapping $\Z^D$ to $\Z$. Let $\g$ be a real nilpotent Lie algebra of dimension $d$ and step $s$. Let $\cX$ be a rational basis for $\g$ which passes through $\g_\bullet$ and with respect to which $\g$ has rational structure constants of height at most $M$. Let $S_\bullet\leq \g_\bullet$ be a sequence of subspaces of complexity at most $M$. If $p$ is $(A,N)$-strongly irrational in $(\g_\bullet,S_\bullet)$ then there are constants $0<c_1,c_2 = O_{d,s,D,t,\Psi}(1)$ such that  $p^\Psi$ is $O(M^{c_1}/A^{c_2})$-equidistributed in $\log G^\Psi$, where $G^\Psi$ is the Leibman group for $G_\bullet, \Psi$.
\end{cor}
\begin{proof}
This follows from our counting lemma Theorem \ref{t:main-quant}, Lemma \ref{l:Ti+hi=gi} and Proposition \ref{p:comp-leib-gpsi}.
\end{proof}

\subsection{Another proof of the Gowers-Wolf conjecture}\label{ss:gw}

In their original `arithmetic regularity lemma and counting lemma' paper, Green and Tao proved the Gowers-Wolf conjecture \cite[Conjecture 2.5]{GW10} for flag systems of linear forms.

\begin{thm}[{\cite[Theorem 1.13]{GT10}}]\label{t:flag-gw}
Let $\Psi = (\psi_1, \ldots, \psi_t)$ be a collection of linear forms each mapping $\Z^D$ to $\Z$ which satisfy the flag condition, and let $s \geq 1$ be an integer such that the polynomials $\psi_1^{s+1},\ldots,\psi_t^{s+1}$ are linearly independent. For $i=1, \ldots t$,  let $f_i:[-N,N] \to \C$ be functions bounded in magnitude by 1 (and defined to be zero outside of $[-N,N]$). For all $\eps >0$ there exists $\delta>0$ such that if $\min_{i} ||f_i||_{U^{s+1}[-N,N]} \leq \delta$, then 
\[\left|\E_{\vect x \in [-N,N]^D} \prod_{i=1}^t f_i(\psi_i(\vect x))\right| \leq \eps.\]  
\end{thm}

The restriction to flag patterns in the above theorem is a consequence of the use of the arithmetic regularity lemma and complementary counting lemma from \cite{GT10}, which itself is restricted to flag patterns. The author \cite{A21} recently resolved the full result by using a linear algebraic trick to reduce to the flag case whereupon the flag arithmetic regularity lemma and counting lemma from \cite{GT10} sufficed.\footnote{The reader might also be interested to consult \cite{M21} where, even more recently, Manners resolved the problem over finite fields with polynomial bounds \cite[Theorem 1.1.5]{M21}.}
 
\begin{thm}[{\cite[Theorem 1.1]{A21}}]\label{t:gw}
Theorem \ref{t:flag-gw} holds without the restriction that $\Psi$ satisfies the flag condition.
\end{thm}

Our arithmetic regularity lemma Theorem \ref{t:strong-arl} and counting lemma Theorem \ref{t:main-quant} mean that original strategy from \cite{GT10} may be applied directly to resolve the Gowers-Wolf conjecture in full generality. We will not reiterate the proof here since it may be fairly cleanly decomposed as a union of what is is already done in \cite[Theorem 1.13]{GT10} and the upcoming Proposition. We do note that along the way one will need to obtain a version of Theorem \ref{t:main-quant} which allow for averages over shifted sublattices of $\Z^D$. This is directly analogous to how \cite[Theorem 3.6]{GT10} differs from \cite[Theorem 8.6]{GT12}. Such a pursuit is routine but technical and tedious; we do not pursue it here to avoid adding further length to this document.

Recall the definition of the sequence of subspaces $(W_i)_{i=1}^s$ given a sequence of subspaces $(S_i)_{i=1}^s$: let $W_1 = S_1$ and iteratively define $W_i = \spa \{S_j, [W_j,W_{i-j}] \text{ for } j=1,\ldots, i-1\}$ for $i=2,\ldots, s$. Recall also the definition of the subalgebras $\h_i$ given a filtration $\g_\bullet$: Definition \ref{d:hi}.

\begin{prop}\label{p:gw-equid}
Suppose $A>M\geq 1$. Suppose that $p$ is $(A,N)$-strongly irrational in $(\g_\bullet,S_\bullet)$ where $\g$ has rational structure constants of height at most $M$ with respect to a $\cX$ which passes through $\g_\bullet$, that the elements of $S_\bullet$ have complexity at most $M$ with respect to $\cX$, and that $\Psi$ has the property that $V^j = \R^t$. Then there are constants $0<c_1,c_2 = O_{d,s,D,t,\Psi}(1)$ such that $p^\Psi$ is $O(M^{c_1}/A^{c_2})$-equidistributed on $\g^\Psi(S_\bullet)\leq \g^t\cong \g \otimes \R^t$, where $\g^\Psi(S_\bullet)$ contains $\g_j\otimes \R^t$. 
\end{prop}
\begin{proof}
The counting lemma Theorem \ref{t:main-quant} and the $(A,N)$-linear irrationality of $p$ yields that $p^\Psi$ is $O(M^{c_1}/A^{c_2})$-equidistributed on $\g^\Psi(S_\bullet)$; we need to show that $\g^\Psi(S_\bullet)$ contains $\g_j \otimes \R^t$. Firstly, $V^j = \R^t$ implies that $V^i = \R^t$ for all $i\geq j$ and so by Proposition \ref{p:gPsi-Sbullet},  $\g^\Psi(S_\bullet)$ contains $\sum_{i=j}^s W_j \otimes \R^t$. Also, the $(A,N)$-filtration irrationality of $p$ ensures by Lemma \ref{l:Ti+hi=gi} that $S_i + \h_i = \g_i$ for all $i$, and so $W_i + \g_{i+1} =\g_i$ for all $i$ by Lemma \ref{l:maximal-Wi}. Thus $\sum_{i=j}^s W_i = \g_j$ and so $\g^\Psi(S_\bullet)$ does indeed contain $\g_j \otimes \R^t$.
\end{proof}

\begin{cor}
Theorem \ref{t:flag-gw} holds without the restriction that $\Psi$ satisfies the flag condition.
\end{cor}
\begin{proof}
Apply the argument from the proof of \cite[Theorem 1.13]{GT10} with our strongly-irrational arithmetic regularity lemma Theorem \ref{t:strong-arl} in place of \cite[Theorem 1.2]{GT10}, and use Proposition \ref{p:gw-equid} in place of the \cite[Theorem 1.11]{GT10}.
\end{proof}

\begin{appendix}
\section{Details on polynomial sequences in Lie groups and Lie algebras}\label{a:group-to-alg}
In this appendix we may use terminology from \cite[Appendix A]{GT10} and \cite[Appendix A]{GT12} without further introduction. As we have throughout this document, we assume here that the dimension $d$ and step $s$ of $\g$ are of size $O(1)$.

\subsection{Smoothness norms}\label{ss:smooth}
Green and Tao introduced the following `smoothness' norms.

\begin{definition}[{\cite[Definition 8.2]{GT12}}]
 Let $N_1,\ldots,N_D\geq 1$ be integers and let $[\vect N]$ denote the set $[N_1]\times \cdots \times [N_D]$. Let $f: \Z^D \to \R/\Z$ be a polynomial map with Taylor expansion  $f(\vect n) = \sum_{\vect j} \alpha_{\vect j} \binom{\vect n }{\vect j}$. Then define 
\[ ||f||_{C^\infty[\vect{N}]} := \sup_{\vect j \ne 0} \prod_{i=1}^D N_i^{j_i}||\alpha_{\vect j}||_{\R/\Z}.\]
\end{definition}

The main theorems in \cite{GT10} and \cite{GT12} are written with respect to these smoothness norms. Note that one may naturally extend this definition to polynomials $f:\R^D \to \R$. It will be convenient for us to work with a slightly different definition. Recall that we use $\cM$ to denote the set of `monomial maps' which take $\vect x = (x_1, \ldots, x_D)$ to a monomial in the variables $\{x_1,\ldots, x_D\}$. The monomial map which takes $\vect x$ to the constant $1$ will be denoted by $1$.

\begin{definition}[Smoothness norms with respect to monomial basis]\label{d:monomial-norms}
Let $f: \R^D \to \R$ have monomial expansion $f(\vect x) = \sum_{m \in \cM} \beta_m m(\vect x)$. Then define 
\[||f||_{C^\infty_\cM[N]^D} := \sup_{1 \ne m \in \cM}N^{\deg m}||\beta_m||_{\R/\Z}.\]
\end{definition}

\begin{lemma}\label{l:change-basis}
Let $f : \R^D \to \R$ be a multivariate polynomial of degree $s$. Then $||s!f||_{C^\infty_\cM[N]^D} \ll_{s,D} ||f||_{C^\infty[N]^D}$.
\end{lemma}
\begin{proof}
Let $\eps = ||f||_{C^\infty[N]^D}$. 

By definition, we may write $f = f_n + f_\eps$, where, when developed with respect to the multinomial basis, the coefficients of $f_n$ are integers and the coefficients of the degree $d$ terms of $f_\eps$ are of magnitude at most $\eps /N^d$.  Writing $f_n$ with respect to the monomial basis yields coefficients in $\frac{1}{s!}\Z$ and so $s!f_n$ has integer valued coefficients (with respect to the monomial basis). Writing $s!f_\eps$ with respect to the monomial basis yields coefficients of degree $d$ terms with magnitude at most $O_{s,D}(\eps /N^d)$.
\end{proof}

\subsection{Equidistribution in the Lie algebra}\label{ss:lie-alg-equid}
The goal of this subsection is to explain what is meant by, and to justify, Theorems \ref{t:leibman-algebra} and \ref{t:gt-algebra}. We deduce them from theorems in \cite{Lei05b} and \cite{GT12} respectively. We will first deal with the qualitative setting (i.e. deducing Theorem \ref{t:leibman-algebra} from \cite{Lei05b}).

When we speak of qualitative equidistribition on a nilmanifold, we mean the following.

\begin{definition}[Qualitative nilmanifold equidistribution] \label{d:nilman-equid}
Let $G/\Gamma$ be a nilmanifold. Recall that this comes with a unique normalised Haar measure, with respect to which we will integrate. A sequence $(g(\vect{n})\Gamma)_{\vect{n} \in \Z^D}$ on $G/\Gamma$ \textit{equidistributes} if for all continuous $F: G/\Gamma \to \C$ we have
\[\lim_{N \to \infty} \E_{\vect n\in [N]^D}F(g(\vect n)\Gamma) \to \int_{G/\Gamma}F. \]
\end{definition}

Leibman's theorem says that a polynomial sequence equidistributes if and only if it does not lie in (a coset of) the kernel of a nontrivial horizontal character on $G$.

\begin{thm}[Leibman's criterion for qualitative equidistribution, \cite{Lei05b}]\label{t:leib}
Let $D$ be a positive integer, let $G/\Gamma$ be a nilmanifold and let $g:\Z^D \to G$ be a polynomial sequence. Then exactly one of the following is true:
\begin{enumerate}
\item $(g(\vect n)\Gamma)_{\vect n\in \Z^D}$ equidistributes in $G/\Gamma$, 
\item there exists a nontrivial horizontal character $\eta: G \to \R/\Z$ which maps $\Gamma$ to $\Z$ and such that $\eta \circ g$ is constant.
\end{enumerate}
\end{thm}

Given two lattices $\Gamma_0$ and $\Gamma_1$ in $G$, we say that they are \textit{commensurable} if $\Gamma_0\cap \Gamma_1$ is a lattice in $G$ of finite index in both $\Gamma_0$ and $\Gamma_1$. It is not difficult to check directly from the above definition  that if $\Gamma_0$ and $\Gamma_1$ are commensurable then a polynomial sequence $g$ equidistributes on $G/\Gamma_0$ if and only if it equidistributes on $G/\Gamma_1$. Therefore, we may extend the definition of equidistribution on a particular nilmanifold $G/\Gamma$ to consider simultaneously all lattices commensurable to $\Gamma$. This may be neatly understood in the Lie algebra $\g$ by the following result. 

\begin{thm}[{\cite[Theorem 5.1.12]{CG90}}]
Let $\Gamma_0,\Gamma_1$ be lattices in $G$. Then $\Gamma_0$ and $\Gamma_1$ determine the same rational structure on $\g$ (i.e. $\spa_\Q(\log \Gamma_0) = \spa_\Q(\log \Gamma_1)$) if and only if $\Gamma_0$ and $\Gamma_1$ are commensurable. 
\end{thm}

We remark also that given any rational structure $\g_\Q$ on $\g$, there is a lattice $\Gamma$ in $G = \exp \g$ (invoke Lie's third theorem)  such that $\g_\Q = \spa_\Q(\log \Gamma)$ (\cite[Theorem 5.1.8]{CG90}). Finally, any polynomial sequence $G$ yields (under the logarithm map) a polynomial sequence in $\g$ and conversely. By the discussion above, we may define the property of equidistribution of a polynomial sequence in $\g$ with respect to a rational structure $\g_\Q$ as that of the corresponding polynomial sequence in $G$ with respect to any lattice $\Gamma$ such that $\spa_\Q(\log \Gamma) = \g_\Q$. 

\begin{definition}[Equidistribution with respect to a rational structure]
A polynomial sequence $p$ in $\g$ \textit{equidistributes in $(\g,\g_\Q)$} if $\exp p$ equidistributes in $G/\Gamma$, where $\Gamma$ is any lattice in $G$ such that $\spa_\Q \log \Gamma = \g_\Q$.
\end{definition}
By our discussion above, this definition is indeed well-defined. We are ready to prove the following. The proof is very straightforward, but we include it to reiterate what is meant by equidistribution in the Lie algebra. Recall that in the following (and indeed in all discussions of irrationality above) we have restricted to polynomial sequences with $p(0) =0$.

\begin{thm}[Leibman's criterion for qualitative equidistribution in the Lie algebra, Theorem \ref{t:leibman-algebra}]\label{t:pr-leibman-algebra}
Let $D$ be a positive integer, let $\g$ be a real, finite dimensional, nilpotent Lie algebra and let $\g_\Q$ be a rational structure on $\g$. Let $(p(\vect n))_{\vect n \in \N^D}$ be a polynomial sequence on $\g$. Then $(p(\vect n))_{{\vect n \in \Z^D}}$ equidistributes in $(\g,\g_\Q)$ if and only if $p$ is additively irrational in $(\g,\g_\Q)$. 

\end{thm}
\begin{proof}
Suppose $p$ is not additively irrational in $(\g,\g_\Q)$, so Definition \ref{d:additive-irrat}, there exists a nontrivial $\eta \in \Hom_\Q(\g,\R)$ such that $((\eta \circ p)(\vect n))_{\vect n \in \Z^D} \subset \Q$. Let $\Gamma$ be a lattice in $G = \exp \g$ such that $\spa_\Q \log \Gamma = \g_\Q$ (cf. \cite[Theorem 5.1.8]{CG90} for its existence). Then $\eta(\log \Gamma) \subset \Q$. In fact, since $\log \Gamma$ is contained in some additive lattice itself contained in $\g_\Q$ (\cite[Theorem 5.4.2]{CG90}), there is a positive integer $M$ such that $\eta(\log \Gamma) \subset \frac{1}{M}\Z$. Thus there exists $\eta'$,  a large integer multiple of $\eta$, so that $\eta'(\log \Gamma) \subset \Z$ and $((\eta' \circ p)(\vect n))_{\vect n \in \Z^D} \subset \Z$. By standard Lie theory there is a unique Lie group homomorphism $\eta_0: G \to \R/\Z$ such that $\eta' = d\eta_0$ and so $\eta_0 \circ \exp = \exp \circ \eta'$ (where here the second $\exp:\R \to \R/\Z$ is just the quotient map). Writing this out we have 
\[ \bar 0 = ((\exp \circ \eta' \circ p)(\vect n))_{\vect n \in \Z^D} =  ((\eta_0 \circ \exp \circ p)(\vect n))_{\vect n \in \Z^D},\]
where $\bar 0$ is the residue class of $0$ in $\R/\Z$.  Note $\eta_0$ is not trivial. Thus the polynomial sequence $\exp \circ p$ does not equidistribute in $G/\Gamma$ by Theorem \ref{t:leib}, and so $p$ does not equidistribute in $(\g,\g_\Q)$ by definition.

In the other direction, suppose that $(p(\vect n))_{{\vect n \in \Z^D}}$ does not equidistribute in $(\g,\g_\Q)$, so $\exp \circ p$ does not equidistribute in $G/\Gamma$ where $\Gamma$ is a lattice such that $\spa_\Q(\log \Gamma) = \g_\Q$. By Theorem \ref{t:leib}, there is a nontrivial horizontal character $\eta_0: G \to \R/\Z$ which maps $\Gamma$ to $\Z$ and such that $\eta_0 \circ \exp \circ p$ is constant in $\R/\Z$. Standard Lie theory yields that the differential $d \eta_0$ is a nontrivial Lie algebra homomorphism $\g \to \R$ and $\eta_0 \circ \exp = \exp \circ d\eta_0$. Set $\eta = d\eta_0$. Thus $\exp \circ \eta \circ p$ is constant (in fact, zero) in $\R/\Z$, so $\eta \circ p \subset \Z$, and $p$ is not additively irrational.

\end{proof}

Now we proceed to the matter of quantitative equidistribution. In the body of the document, we used the notation $\langle \cX \rangle$ to refer to the additive group generated by a basis $\cX$, that is, $\spa_\Z \cX$. In this appendix we will also have to deal with $\langle \cX \rangle_\ast$, the $\ast$-multiplicative group generated by $\cX$, and so we will generally use $\spa_\Z \cX$ to refer to the additive group in order to make the distinction more notationally clear. 

 In the following definition we refer to a Lipschitz function on $G/\Gamma$, which obviously presupposes the existence of a metric on this space. We refer the reader to \cite{GT12} for details on this; they are not important for the purposes of this document. 

\begin{definition}[Quantitative nilmanifold equidistribution]\label{d:quant-equid-grp}
Let $0 < \delta < 1/2$, let $G/\Gamma$ be a nilmanifold, let $N_1,\ldots,N_D\geq 1$ and let $[\vect N]$ denote the set $[N_1]\times \cdots \times [N_D]$. We say that a polynomial sequence $(g(\vect n)\Gamma)_{\vect n \in \Z^D}$ is \textit{$\delta$-equidistributed} in $G/ \Gamma$ if for all Lipschitz functions $F:G/\Gamma \to \C$ we have 
\[|\E_{\vect n\in [\vect N]} F(g(\vect n)\Gamma) - \int_{G/\Gamma}F| \leq \delta ||F||_{\text{Lip}}.\] 
\end{definition}

Green and Tao proved the following; \cite{GT12} may be consulted for terminology/notation which has not otherwise been introduced in this document. 

\begin{thm}[{\cite[Theorem p.3]{GT15}}]\label{t:gt}
Let $0<\delta<1/2$, let $N_1,\ldots,N_D\geq 1$ and let $[\vect N]$ denote the set $[N_1]\times \cdots \times [N_D]$. Let $G/\Gamma$  be a nilmanifold with dimension $d$, step $s$ and Mal'cev basis $\cX$ which is $1/\delta$-rational. Let $(g(\vect n))_{\vect n \in \Z^D}$ be a polynomial sequence in $G$.  Then either
\begin{enumerate}
\item there is some $N_i\ll \delta^{-O_{d,s,D}(1)}$, or
\item $(g(\vect n)\Gamma)_{\vect n \in [\vect N]}$ is $\delta$-equidistributed in $G/\Gamma$, or
\item there exists a nontrivial horizontal character $\eta$ on $G/\Gamma$ with $0<||\eta|| \ll \delta^{-O_{d,s,D}(1)}$ such that $||\eta \circ g||_{C^\infty[\vect N]} \ll \delta^{-O_{d,s,D}(1)}$. 
\end{enumerate}
Furthermore, when $N_1 = \cdots = N_D$, the first option above may be removed. 
\end{thm}

Our first task in the quantitative setting is to define equidistribution of a polynomial sequence in the Lie algebra with respect to a choice of (rational) basis.

\begin{definition}[Definition \ref{d:quant-equid-alg}, Quantitative equidistribution in the Lie algebra]\label{d:quant-equid-alg-ap}
Let $\cX$ be a rational basis in $\g$. Then a polynomial sequence $(p(\vect n))_{\vect n \in [\vect N]}$ is $\delta$-equidistributed in $(\g, \cX)$ if $(\exp (p(\vect n)))_{\vect n \in [\vect N]}$ is $\delta$-equidistributed in $G/\Gamma$, where $\Gamma$ is the smallest multiplicative lattice containing $\spa_\Z \cX$.
\end{definition}

Let us check that this definition is sensible. Firstly, the intersection of two multiplicative lattices that contain $\cX$ is clearly a discrete multiplicative group which spans $\g$; this is enough to conclude \cite[Corollary 5.4.5]{CG90} that it is cocompact in $G$ and so indeed itself a multiplicative lattice; thus we may indeed speak of the smallest lattice containing $\spa_\Z \cX$.

Next, in the context of constructing nilsequences, a function on $\g$ is of course multiplicatively automorphic with respect to the set $\spa_\Z \cX$ if and only if it is multiplicatively automorphic with respect to $\langle \spa_\Z \cX \rangle_\ast$. 

The following lemma offers some further validation for the above definition.

\begin{lemma}\label{l:mult-closure}
Let $\Lambda$ be an additive lattice in $\g$. Then the smallest $\ast$-multiplicative subgroup $\langle \Lambda \rangle_\ast$ containing $\Lambda$ is a multiplicative lattice in $\g$. Furthermore, if $\cX$ is a $\Z$-basis for $\Lambda$ which has rational structure constants of height at most $M$, then we have 
\[\Lambda \subset \langle \Lambda \rangle_\ast \subset \frac{1}{M^{O(1)}}\Lambda, \]
where implicit constants may depend on $\dim \g$.
\end{lemma}
\begin{proof}
Firstly, it is not difficult to construct a basis $\tilde \cX := \{\tilde X_1, \ldots, \tilde X_d\}$ for $\g$ consisting of rational linear combinations of $\cX$ with coefficients of height $M^{O(1)}$ which satisfies the nesting property 
\[[\g, \tilde X_i] \subset \spa_\R\{\tilde X_{i+1}, \ldots, \tilde X_d\};\]
see \cite[Proposition A.9]{GT12} for details. Thus we may assume without loss of generality that $\cX$ satisfies this nesting property; denote $\cX := \{ X_1, \ldots, X_d\}$. We claim that there are integers $k_1, \ldots, k_d$ each of size $M^{O(1)}$ such that upon setting $\cX':= \{\frac{1}{k_1}X_1, \ldots, \frac{1}{k_d}X_d\}$ we have that $\spa_\Z\cX'$ is a multiplicative lattice in $\g$. Then 
\[\Lambda \subset \langle \Lambda \rangle_\ast \subset \spa_\Z{\cX'} \subset \frac{1}{M^{O(1)}}\Lambda,\]
and so we are done.

It follows from Baker-Campbell-Hausdorff that multiplication in  coordinates with respect to $\cX$ is a polynomial map $\ast: \R^d \times \R^d \to \R^d$ of degree $O(1)$ and with rational coefficients of height $M^{O(1)}$. Let $P_i$ be polynomials such that $a \ast b = (P_1(a,b), \ldots, P_d(a,b))$, with respect to $\cX$. Then by Baker-Campbell-Hausdorff and the nesting property of $\cX$, we have that
\[ P_i(a,b) =  a_i + b_i + \tilde P_i(a_1, \ldots, a_{i-1},b_1,\ldots, b_{i-1}),\]
where $\tilde P_i$ is a polynomial with all terms of degree at least 2.  

 Let $k_1, \ldots, k_d$ be positive integers and form $\cX'$ as above. Then we have that 
\[P'_i(a,b) = k_iP_i(a_1/k_1,\ldots, a_d/k_d, b_1/k_1,\ldots,b_d/k_d),\]
where the $P_i'$ are the polynomials which define multiplication in $\g$ with respect to $\cX'$. In particular, defining $\tilde P_i'$ analgously, we have 
\[\tilde P'_i(a,b) = k_i\tilde P_i (a_1/k_1,\ldots, a_{i-1}/k_{i-1}, b_1/k_1,\ldots, b_{i-1}/k_{i-1}).\]
It follows that we may inductively choose $k_i$ of size $M^{O(1)}$ depending on $P_i$ and $(k_1, \ldots, k_{i-1})$ so that $\tilde P_i'$ has integer coefficients and thus so too does $P_i'$. Choosing the $k_i$ in this way we have $\spa_\Z \cX' \ast \spa_\Z \cX' \subset \spa_\Z \cX'$. Also, $\spa_\Z\cX'$ is clearly closed under inverses, and so is a multiplicative subgroup. Furthermore, $\spa_\Z \cX'$ is clearly discrete and its cocompactness in $\g$ follows from \cite[Theorem 5.1.6]{CG90}. This completes the proof.
\end{proof}

\begin{lemma}\label{l:near-malcev}
Let $\Lambda$ be an additive lattice in $\g$ with $\Z$-basis $\cX$, which has rational structure constants of height at most $M$. Then the multiplicative lattice $\langle \Lambda \rangle_\ast$ has a Mal'cev basis whose elements are rational linear combinations of $\cX$ with coefficients of height $M^{O(1)}$.
\end{lemma}
\begin{proof}
The existence of a Mal'cev basis for $\langle \Lambda \rangle_\ast$ follows from Lemma \ref{l:mult-closure} and \cite[Theorem 5.1.6]{CG90}. It is not difficult to use (the proofs of) these results to recover the quantitative statement of this lemma. We omit the details. 
\end{proof}

Deducing the Lie algebra version of Green and Tao's criterion for quantitative equidistribution from its Lie group predecessor is similar to what we have already done for Leibman's criterion in the qualitative setting. We will only sketch the proof below.

\begin{thm}[Green and Tao's criterion for quantitative equidistribution in the Lie algebra, Theorem \ref{t:gt-algebra}]\label{t:pr-gt-algebra}
Let $0<\delta<1/2$, let $N_1,\ldots,N_D\geq 1$ and let $[\vect N]$ denote the set $[N_1]\times \cdots \times [N_D]$. Let $\g$  be a  real nilpotent Lie algebra with dimension $d$, step $s$ and basis $\cX$ which is $1/\delta$-rational. Let $(p(\vect n))_{\vect n \in \Z^D}$ be a polynomial sequence in $\g$.  There are positive constants $c_{d,s,D}, c'_{d,s,D}$ such that at least one of the following is true:
\begin{enumerate}
\item there is some $N_i< \delta^{-c'_{d,s,D}}$, 
\item $(p(\vect n))_{{\vect n \in [\vect N]}}$ is $\delta$-equidistributed in $(\g, \cX)$, 
\item $p$ is not $(\delta^{-c_{d,s,D}}, \vect N)$-additively irrational in $(\g,\cX)$. 
\end{enumerate}
Furthermore, when $N_1 = \cdots = N_D$, the first option above may be removed. 
\end{thm}
\begin{proof}
Let all constants depend on $s, d, D$. Suppose that (1) and (2) do not hold and denote $G = \exp \g$ (as above, we are implicitly invoking Lie's third theorem) and $\Gamma  = \exp (\langle \spa_\Z \cX \rangle_\ast)$ so that, by definition, $\exp \circ p$ is not $\delta$-equidistributed in $G/\Gamma$. Let $\cX_{\mal}$ be the Mal'cev basis for $\Gamma$ which is produced by Lemma \ref{l:near-malcev}; it is $\delta^{-O(1)}$-rational. By Theorem \ref{t:gt}, there is a horizontal character $\eta_0$ on $G/\Gamma$ with $0< ||\eta_0|| < \delta^{O(1)}$ such that $||\eta_0 \circ g||_{C^\infty[\vect N]} < \delta^{-O(1)}$. Then one may argue using basic Lie theory as in the qualitative setting to find from $\eta_0$ some nontrivial $\eta \in \Hom_{\cX_\mal}(\g,\R)$ with complexity (as in our Definition \ref{d:comp-map}) $\delta^{-O(1)}$ such that $||\eta \circ p ||_{C^\infty[\vect N]} < \delta^{-O(1)}$. But the elements of $\cX_\mal$ are rational linear combinations of $\cX$ of height $\delta^{-O(1)}$, so there is some $C=\delta^{-O(1)}$ such that $\eta' := C\eta$ is a nontrivial element of $\Hom_\cX(\g,\R)$\footnote{Recall that $\Hom_\cX(\g,\R)$ denotes the set of Lie algebra homomorphisms from $\g$ to $\R$ which map $\cX$ to $\Z$.} of complexity $\delta^{-O(1)}$. Thus we have 
\[||\eta' \circ p ||_{C^\infty[\vect N]} \leq C||\eta \circ p ||_{C^\infty[\vect N]}< \delta^{-O(1)}, \]
and so $p$ is not $(\delta^{-O(1)}, \vect N)$-additively irrational in $(\g,\cX)$.
\end{proof}

\subsection{Rational and small polynomial sequences} 

\begin{lemma}[Smallness and rationality under products]\label{l:poly-products}
Let $\cX$ be a rational basis for $\g$ such that the Lie bracket  has rational structure constants of height at most $M$ with respect to $\cX$. Let $e, e'$ be polynomial sequences adapted to $\g_\bullet$ which are $(A,N)$-small and let $r, r'$ be polynomial sequences adapted to $\g_\bullet$ which are $A$-rational. Then $e \ast e'$ and $r \ast r'$ are polynomial sequences adapted to $\g_\bullet$, the former is $((AM)^{O(1)},N)$-small and the latter is $(AM)^{O(1)}$-rational.  
\end{lemma}
\begin{proof}
That $e \ast e'$ and $r \ast r'$ are polynomial sequences adapted to $\g_\bullet$ is standard. The rest follows from definitions and some easy, but disproportionately tedious, arguments. We omit the details. 
\end{proof}

\begin{lemma}\label{l:rat-period}
Suppose that $r$ is a polynomial sequence in $\g$ which is $A$-rational with respect to $\cX$. Suppose that $\g$ has rational structure constants of height at most $M$. Then for any $\ast$-multiplicative subgroup $\Gamma$ of $\g$ which contains $\langle \cX \rangle$, the orbit $r(n) \ast \Gamma$ is $(AM)^{O(1)}$-periodic.
\end{lemma}
\begin{proof}
We show that there is an integer $q$ of size $(AM)^{O(1)}$ such that $p(n)^{-1}\ast p(n+q)\in \spa_\Z \cX$ for all integers $n$. For the meantime let $q$ be undetermined. Note $p(n)^{-1} = -p(n)$ and let $p'$ be the polynomial sequence in 2 variables such that $qp'(n,q) = p(n+q)-p(n)$. By Baker-Campbell-Hausdorff we have 
\[p(n)^{-1}\ast p(n+q) = p(n+q) - p(n) - \frac{1}{2}[p(n),p(n+q)] + \cdots.\] 
We also have from definitions that $[p(n),p(n+q)] = q[p(n),p'(n,q)]$, and indeed every term in the Baker-Campbell-Hausdorff expansion will be a polynomial in $n$ and $q$ which is divisible by $q$ and which has coefficients in $(\frac{1}{(AM)^{O(1)}}\Z)^{\deg \g}$, so that ultimately the polynomial $p(n)^{-1}\ast p(n+q)$ has this property. Therefore, a specific value of $q$ may be chosen of size at most $(AM)^{O(1)}$ so that, for all $n \in \Z$, $p(n)^{-1}\ast p(n+q) \in \spa_\Z \cX$. 
\end{proof}

\subsection{Factorising for filtration irrationality} \label{a:fact}
In this section we sketch how to deduce Proposition \ref{p:irrational-factorise} from  \cite[Lemma A.10]{GT10}. It is mostly a matter of transferring between related notions of complexity whose details we largely omit. The main additional ingredient is the following.

\begin{lemma}\label{l:add-contains-mult}
Let $M>1$ be a positive integer. Let $\g$ have rational basis $\cX$ which has rational structure constants of height at most $M$. Then there exists a \textit{multiplicative} lattice $\Gamma \leq G$ such that $M^{O(1)} \spa_\Z \cX \subseteq \log \Gamma \subseteq \frac{1}{M^{O(1)}} \spa_\Z \cX$.\footnote{In the language of \cite{GT12}, $\cX$ is then a \textit{$M^{O(1)}$-rational weak basis} for $G/\Gamma$.}
\end{lemma}
\begin{proof}
Firstly we argue that there is a strong Mal'cev basis\footnote{Here our usage of `strong Mal'cev basis' is as in \cite{CG90}: a basis $\{X_1, \ldots, X_d\}$ for $\g$ such that $\spa_\R\{X_i, \ldots, X_d\}$ is an ideal for all $i$.} $\cX'$ such that \begin{equation}\label{e:basis-contain}
M^{O(1)} \spa_\Z \cX' \subseteq \spa_\Z\cX \subseteq \frac{1}{M^{O(1)}} \spa_\Z \cX'.
\end{equation} 
This is equivalent to finding a strong Mal'cev basis whose elements may be written as rational linear combinations of $\cX$ with coefficients of height at most $M^{O(1)}$. Once one makes the observation that the elements of the descending central series have complexity $M^{O(1)}$ in $\g$ with respect to $\cX$ as a consequence on the bound on the height of the rational structure constants, the rest of the claim is established in the proof of \cite[Proposition A.9]{GT12}. 

Next, by working through the details of \cite[Theorem 5.1.8(b)]{CG90} (in particular, making note of quantitative bounds), one obtains that there is an integer $K = M^{O(1)}$ such the set $\Z^d \subset \R^d$ in \textit{Mal'cev coordinates}\footnote{Our usage of `Mal'cev coordinates' is as in \cite{GT12}. These coordinates are also known in the literature as \textit{coordinates of the second kind}.}  with respect to $K\cX'$ is a multiplicative lattice (and so in fact $K\cX'$ is a Mal'cev basis for this lattice). Let $\Gamma$ be this lattice. Then \cite[Lemma A.2]{GT12} yields that 
\[M^{O(1)}\spa_\Z K\cX' \subseteq \log \Gamma \subseteq \frac{1}{M^{O(1)}}\spa_\Z K\cX' .\]
Putting this together with Equation (\ref{e:basis-contain}) proves the lemma.

\end{proof}

Next we restate \cite[Lemma 2.10]{GT10} from which we will derive Proposition \ref{p:irrational-factorise}. The statement is copied verbatim (including notation) from \cite{GT10} and is included here for the convenience of the reader.

\begin{lemma}[{\cite[Lemma 2.10]{GT10}}]\label{l:gt-fact}
Let $(G/\Gamma, G_\bullet)$ be a degree $\leq s$ filtered nilmanifold of complexity $\leq M_0$, and let $g \in \poly(\Z,G_\bullet)$. For any growth function $\cF'$, we can find a quantity $M_0 \leq M \leq O_{M,\cF'}(1)$ and a factorisation $g=\beta g' \gamma$ where: 
\begin{itemize}
\item[] $\beta \in \poly(\Z,G_\bullet)$ is $(O_M(1),N)$-smooth;
\item[] $g' \in \poly(\Z,G_\bullet)$ is $(\cF'(M),N)$-irrational in a subnilmanifold $(G'/\Gamma',G_\bullet')$ of $(G/\Gamma,G_\bullet)$ of complexity $O_M(1)$, and
\item[] $\gamma \in \poly(\Z,G_\bullet)$ is $O_M(1)$-periodic.
\end{itemize}
\end{lemma}

We are ready to prove Proposition \ref{p:irrational-factorise}.

\begin{prop}[Proposition \ref{p:irrational-factorise}]\label{p:pr-irrational-factorise}
Let $M_0>1$. Let $\g$  be a real nilpotent Lie algebra of dimension at most $M_0$  with filtration $\g_\bullet$ of length at most $M_0$. Let $\cX$ be a rational basis for $\g$ with structure constants of height at most $M_0$. Let $p$ be a polynomial sequence adapted to $\g_\bullet$  with $p(0)=0$. For any growth function $\cF: \R^+ \to \R^+$, there exists $M_0 \leq M \leq O_{M_0,\cF}(1)$ and a factorisation $p = e \ast p' \ast r$ where $e$ is $(O_M(1),N)$-small, where $r$ is $O_M(1)$-rational and where $p'$ is $(\cF(M),N)$-filtration irrational in a subfiltration $\g'_\bullet \leq \g_\bullet$ of complexity $O_M(1)$.\footnote{That is, for all $i$, $\g_i'$ is a subspace of $\g_i$ of complexity $O_M(1)$.} This filtration irrationality is measured with respect to $\cX'$, a Mal'cev basis\footnote{See \cite[Definition 2.1]{GT12}.} for $\g_\bullet'$ which may be written as a linear combination of $\cX$ with rational coefficients of height $O_M(1)$..
\end{prop}
\begin{proof}
We sketch the argument as a detailed proof is both straightforward and tedious. We refer the reader to \cite[Appendix A]{GT12} for more details on Mal'cev bases and quantitative aspects thereof. We will use terminology from there without further introduction.

Let $\Gamma$ be the multiplicative lattice produced by Lemma \ref{l:add-contains-mult} with input $\cX$ and $M_0$. Then $G/\Gamma$ will have complexity $M_0^{O(1)}$ as a nilmanifold (cf. \cite[Definition 1.4]{GT10}).\footnote{This is not entirely trivial as one needs to construct from $\cX$ a Mal'cev basis for $\Gamma$ to exhibit the complexity of $G/\Gamma$. This is done in \cite[Lemma A.8, Proposition A.9]{GT12} together with our Lemma \ref{l:add-contains-mult} which yields that $\cX$ is a $M^{O(1)}$-rational weak basis for $\Gamma$ as per \cite[Definition A.7]{GT12}. Alternatively, one may go inside the proof of Lemma \ref{l:add-contains-mult} to see that the basis $K\cX'$ is in fact an appropriate Mal'cev basis for $\Gamma$.} Invoke Lemma \ref{l:gt-fact} for the polynomial sequence $g = \exp \circ p$ and with lattice $\Gamma$. One may go inside the proof of \cite[Lemma A.10]{GT10} and also use \cite[Lemma A.2]{GT12} to show that the $(O_M(1),N)$-smooth (resp. $O_M(1)$-periodic) sequence produced by \cite[Lemma A.10]{GT10} may instead by taken to be $(O_M(1),N)$-small (resp. $O_M(1)$-rational) by our definition.\footnote{With respect to the Mal'cev basis used in \cite[Lemma A.10]{GT10}, but then the same is true with respect to $\cX$ since the two bases are related by rational linear combinations of height $M_0^{O(1)}$.} Inferring the complexity of $\g'_\bullet:= (\log G'_i)_{i=1}^s$ from that of $G'_\bullet$ is an easy exercise in expanding definitions. Furthermore, the $(\cF(M),N)$-filtration irrationality of $p'$ with respect to an appropriate basis in $\g'$ may be inferred from  \cite[Proposition A.10]{GT12} and the irrationality of $g'$ in $G'/(G'\cap \Gamma)$.
\end{proof}

\section{Some quantitative integer linear algebra}\label{a:lin-alg}
In this appendix we define what a Hermite basis is and collect some (very standard) results that one might use to prove Lemmas \ref{l:comp-lift-restrict} and \ref{l:comp-operations}, which are essentially all of the form `integer linear algebra engenders at most polynomial growth in the size of the integers with which we are computing'. 

\subsection{Hermite normal form}\label{ss:hnf}

Given a vector space $\g$ with basis $\cX$ and a rational subspace $S$, we want to define a basis $\cX'$ for $S$ such that $\spa_\Z \cX' = (\spa_\Z \cX) \cap S$. To do so, we will use the theory of the Hermite normal form. Everything in this subsection is completely standard and is included here for the reader's convenience.

Let $m,n$ be positive integers and let $M$ be a $m\times n$ integer valued matrix. 

\begin{thd}
The \textit{Hermite normal form} of $M$ is the unique matrix $H$ of the form $H=MU$ where $U$ is a unimodular matrix and $H$ has the following properties: 
\begin{enumerate}
\item $H$ is lower triangular,
\item for any column, the highest nonzero entry (the `pivot') is strictly below that of all columns to its left,
\item all pivots are positive,
\item for all rows with a pivot, all elements to the right of the pivot are zero and all elements to the left of the pivot are non-negative and strictly smaller than the pivot. 
\end{enumerate}
\end{thd}

For $\g,\cX, S$ as above, the rationality of $S$ in $\g$ with respect to the rational structure induced by $\cX$ yields that $\spa_\Z\cX \cap S$ is a full rank lattice in $S$. Let $\cZ$ be any $\Z$-basis for this lattice and let $M_\cZ$ be the $\dim \g \times \dim S$ integer matrix whose columns are the elements of $\cZ$ with respect to $\cX$. Let $H$ be the Hermite normal form of $M_\cZ$. The above Definition/Theorem yields that in fact $H$ is independent of the choice of basis $\cZ$, since any other $M_{\cZ'} = M_{\cZ}U$ for some unimodular $U$. Thus, to any $\cX, S$ we may associate a unique matrix $H_{\cX,S}$ in Hermite normal form whose columns are a $\Z$-basis for $\spa_\Z \cX \cap S$. 

\begin{definition}[Hermite basis]\label{d:hermite-basis}
Let $\g, \cX, S, H_{\cX,S}$ be as above. We define the \textit{Hermite basis} $\cX'_{\cX,S}$ to be the (ordered) set of columns of the Hermite normal form $H_{\cX,S}$ with respect to $\cX$.
\end{definition}

The following lemma is a corollary of any sensible algorithm to compute the Hermite normal form.

\begin{lemma}\label{l:herm-poly}
Let $m,n=O(1)$. Let $M$ be a $m\times n$ matrix with integer entries of size at most $C$. Then the Hermite normal form of $M$ has entries of size $C^{O(1)}$.
\end{lemma}

As a result, given bounds on the complexity of a subspace, one may bound the size of the entries of its Hermite basis.

The following lemma is essentially all of the linear algebra that is needed to prove Lemmas \ref{l:comp-lift-restrict} and \ref{l:comp-operations}. We omit a proof. 

\begin{lemma}\label{l:lin-alg}
Let $m,n = O(1)$. Let $M$ be a $m\times n$ full rank matrix with integer entries of size at most $C$. Then $\ker M \cap \Z^n$ is a lattice which has a $\Z$-basis consisting of integer vectors whose elements have size at most $C^{O(1)}$.

Conversely, given a $\Z$-basis for a rank-$m$ integer lattice $L$ in $\R^n$ whose vectors have elements of size at most $C$, there is an $m\times n$ integer matrix $M$ whose entries have size at most $C^{O(1)}$ such that $\ker M \cap \Z^n = L$.
\end{lemma}

\end{appendix}

\bibliographystyle{alpha}
\bibliography{cl}

\begin{thebibliography}{Man21}

\bibitem[Alt22]{A21}
D.~Altman.
\newblock On a conjecture of {G}owers and {W}olf.
\newblock {\em Discrete Analysis}, 10:13, 2022.

\bibitem[CG90]{CG90}
L.~J. Corwin and F.~P. Greenleaf.
\newblock {\em Representations of nilpotent {L}ie groups and their
  applications. {P}art {I}}, volume~18 of {\em Cambridge Studies in Advanced
  Mathematics}.
\newblock Cambridge University Press, Cambridge, 1990.
\newblock Basic theory and examples.

\bibitem[Gre05]{G05}
B.~Green.
\newblock A {S}zemer\'{e}di-type regularity lemma in abelian groups, with
  applications.
\newblock {\em Geom. Funct. Anal.}, 15(2):340--376, 2005.

\bibitem[GT10]{GT10}
B.~Green and T.~Tao.
\newblock An arithmetic regularity lemma, an associated counting lemma, and
  applications.
\newblock In {\em An irregular mind}, volume~21 of {\em Bolyai Soc. Math.
  Stud.}, pages 261--334. J\'{a}nos Bolyai Math. Soc., Budapest, 2010.

\bibitem[GT12]{GT12}
B.~Green and T.~Tao.
\newblock The quantitative behaviour of polynomial orbits on nilmanifolds.
\newblock {\em Ann. of Math. (2)}, 175(2):465--540, 2012.

\bibitem[GT15]{GT15}
B.~Green and T.~Tao.
\newblock On the quantitative distribution of polynomial nilsequences -
  erratum, 2015.
\newblock \url{arxiv:1311.6170}.

\bibitem[GT20]{GT20}
B.~Green and T.~Tao.
\newblock An arithmetic regularity lemma, associated counting lemma, and
  applications, 2020.
\newblock \url{arXiv:1002.2028v3}.

\bibitem[GW10]{GW10}
W.~T. Gowers and J.~Wolf.
\newblock The true complexity of a system of linear equations.
\newblock {\em Proc. Lond. Math. Soc. (3)}, 100(1):155--176, 2010.

\bibitem[Lei02]{Lei02}
A.~Leibman.
\newblock Polynomial mappings of groups.
\newblock {\em Israel J. Math.}, 129:29--60, 2002.

\bibitem[Lei05]{Lei05b}
A.~Leibman.
\newblock Pointwise convergence of ergodic averages for polynomial actions of
  {${\Bbb Z}^d$} by translations on a nilmanifold.
\newblock {\em Ergodic Theory Dynam. Systems}, 25(1):215--225, 2005.

\bibitem[Man21]{M21}
F.~Manners.
\newblock True complexity and iterated {C}auchy--{S}chwarz, 2021.
\newblock \url{arXiv:2109.05731}.

\bibitem[Sze75]{Sz75}
E.~Szemer\'{e}di.
\newblock On sets of integers containing no {$k$} elements in arithmetic
  progression.
\newblock {\em Acta Arith.}, 27:199--245, 1975.

\bibitem[Tao20]{T20}
T.~Tao.
\newblock A correction to ``an arithmetic regularity lemma, an associated
  counting lemma, and applications'', 2020.
\newblock
  \url{https://terrytao.wordpress.com/2020/11/26/a-correction-to-an-arithmetic-regularity-lemma-an-associated-counting-lemma-and-applications/}.

\end{thebibliography}

\end{document}